\documentclass[12pt, oneside, reqno]{amsart}
\usepackage{amscd}
\usepackage{bbm}
\usepackage{dsfont}
\usepackage{enumerate}
\usepackage{latexsym}
\usepackage{srcltx}
\usepackage{amsfonts}
\usepackage{amssymb}
\usepackage{datetime}
\usepackage{setspace}
\usepackage{enumerate}
\usepackage{amsthm}
\usepackage{graphicx}
\usepackage{epstopdf}
\usepackage{amsmath}

\newtheorem{theorem}{Theorem}[section]
\newtheorem{proposition}[theorem]{Proposition}
\newtheorem{lemma}[theorem]{Lemma}
\newtheorem{corollary}[theorem]{Corollary}

\theoremstyle{definition}
\newtheorem{example}[theorem]{Example}
\newtheorem{definition}[theorem]{Definition}

\newtheorem{remark} [theorem] {Remark}

\newtheorem{problem} [theorem] {Problem}

\input epsf
\begin{document}
\title{ Spectrum of Weighted Composition Operators \\
Part V \\
Spectrum and essential spectra of weighted rotation-like operators.}

\author{Arkady Kitover}

\address{Community College of Philadelphia, 1700 Spring Garden St., Philadelphia, PA,
USA}

\email{akitover@ccp.edu}

\author{Mehmet Orhon}

\address{University of New Hampshire, Durham, NH, 03824}

\email{mo@unh.edu}


\date{\today}


\maketitle

\markboth{Arkady Kitover and Mehmet Orhon}{Spectrum of weighted composition operators.
V}

Dedicated to Eric Nordgren and to the memory of Herbert Kamowitz

\begin{abstract}
 We introduce the class of weighted "rotation-like" operators and study general properties of essential spectra of such operators. Then we use this approach to investigate and in some cases completely describe essential spectra of weighted rotation operators in Banach spaces of measurable and analytic functions.
\end{abstract}

\section{Introduction}

This paper is a continuation of the study of the spectrum and the essential spectra of
weighted
composition operators undertaken by the first named author in~\cite{Ki1} - ~\cite{Ki4},
see also~\cite{AAK}.

To illustrate what the paper is about let us consider the following example. Let $X$ be
a Banach ideal space ( see Definition~\ref{d1} below) of Lebesgue measurable functions
on the unit circle such that $L^\infty \subseteq X \subseteq L^1$, and the norm on $X$
is rotation invariant. Consider a non-periodic rotation of the circle. Let $U$ be the
corresponding composition operator on $X$ and $T = wU$ where $w \in L^\infty$. Then it
follows from the results in~\cite{Ki1} - ~\cite{Ki3} that the spectrum of $T$ is a
connected rotation-invariant subset of the complex plane, and the essential spectra of
the operator $T$ coincide with its spectrum. Such a simple description of spectra is
due to two circumstances:

(1) Non-periodic rotations are ergodic.

(2) The composition operator $U$ is an invertible isometry on $X$.

\noindent But it is another property of $U$, not used in~\cite{Ki1} - ~\cite{Ki3}, we
are especially interested in here:

(3) $UM_zU^{-1} = \gamma M_z$ where $M_z$ is the multiplication operator,
$M_zx(e^{i\theta}) = e^{i\theta}x(e^{i\theta}), x \in X, \theta \in [0,2\pi]$,
$|\gamma|=1$, and $\gamma$ is not a root of unity.

Property (3) gives the rise to the definition of rotation-like operators introduced and
studied below in Section 3.

The structure of the paper is as follows. In Section 2 we introduce the notations,
recall the basic definitions, and state some known results needed in the sequel. In
Section 3 we introduce the notion of a "rotation-like" operator and prove some results
about the spectrum and essential spectra of weighted "rotation-like" operators. In
sections 4, 5, and 6 we discuss the weighted rotations and/or weighted "rotation-like" operators acting on Banach ideal spaces, spaces of analytic functions, and on uniform algebras, respectively,
and prove some general results concerning the essential spectra of such operators.
Finally, in Section 7 we apply the results from Sections 3 - 6, as well as
from~\cite{Ki1} -~\cite{Ki4} to study the essential spectra of weighted "rotation-like" operators in various Banach spaces of analytic functions.

It is worth noticing that while weighted rotations have some properties that greatly
simplify the study of their spectra, there are many instances when our knowledge of
these spectra remains, at best, incomplete. We highlighted the corresponding questions
by putting them as open problems in the text of the paper.

\bigskip

\section{Preliminaries}

In the sequel we use the following standard notations

\noindent $\mathds{N}$ is the semigroup of all natural numbers.

\noindent $\mathds{Z}$ is the ring of all integers.

\noindent $\mathds{R}$ is the field of all real numbers.

\noindent $\mathds{C}$ is the field of all complex numbers.

\noindent $\mathds{T}$ is the unit circle. We use the same notation for the unit circle
considered as a subset of the complex plane and as the group of all complex numbers
with modulus 1.

\noindent $\mathds{U}$ is the open unit disc.

\noindent $\mathds{D}$ is the closed unit disc.

\noindent For any $n >1$ we denote by $\mathds{U}^n$, $\mathds{T}^n$, and
$\mathds{B}^n$ the open unit polydisc, the unit torus, and the open unit ball in
$\mathds{C}^n$, respectively.

All the linear spaces are considered over the field $\mathds{C}$ of complex numbers.

Let $\Omega$ be an open subset of $\mathds{C}^n$. We denote by $\mathcal{H}(\Omega)$
the vector space of all functions analytic in $\Omega$ with the topology of uniform
convergence on compact subsets of $\Omega$.

The algebra of all bounded linear operators on a Banach space $X$ is denoted by $L(X)$.

Let A be a commutative unital Banach algebra. We denote by $\mathfrak{M}_A$ and by
$\partial A$ the space of maximal ideals and the Shilov boundary of $A$, respectively.

Let $E$ be a set, $\varphi : E \rightarrow E$ be a bijection, and $w$ be a
complex-valued function on $E$. Then

\noindent $\varphi^n$ , $n \in \mathds{N}$, is the $n^{th}$ iteration of $\varphi$,

\noindent $\varphi^0(e) = e, e \in E$,

\noindent $\varphi^{-n}$, $n \in \mathds{N}$,  is the $n^{th}$ iteration of the inverse
map $\varphi^{-1}$,

\noindent $w_0 = 1$, $w_n = w(w \circ \varphi) \ldots (w \circ \varphi^{n-1})$, $n \in
\mathds{N}$.

Recall that an operator $T \in L(X)$ is called \textit{semi-Fredholm} if its range
$R(T)$ is closed in $X$ and either $\dim{\ker{T}}< \infty$ or codim $R(T) < \infty$.

The \textit{index} of a semi-Fredholm operator $T$ is defined as

\centerline{ ind $T$ = $  \dim{\ker{T}}$ - $\mathrm{codim} \, R(T)$.}

The subset of $L(X)$ consisting of all semi-Fredholm operators is denoted by $\Phi$.

$\Phi_+ = \{T \in \Phi : \dim{\ker{T}}< \infty\}$.

$\Phi_- = \{T \in \Phi : \textrm{codim}\; R(T) < \infty\}$.

$\mathcal{F} = \Phi_+  \cap  \Phi_-$ is the set of all Fredholm operators in $L(X)$.

$\mathcal{W} = \{T \in \mathcal{F} : \textrm{ind} \; T = 0\}$ is the set of all Weyl
operators in $L(X)$.

Let $T$ be a bounded linear operator on a Banach space $X$. As usual, we denote the
spectrum of $T$ by $\sigma(T)$ and its spectral radius by $\rho(T)$.

We will consider the following subsets of $\sigma(T)$.

$\sigma_p(T) = \{\lambda \in \mathds{C} : \exists x \in X \setminus \{0\}, Tx = \lambda
x\}.$

$\sigma_{a.p.}(T) = \{\lambda \in \mathds{C}: \exists x_n \in X, \|x_n\| = 1, Tx_n -
\lambda x_n \rightarrow 0\}$.

$\sigma_r(T) = \sigma(T) \setminus \sigma_{a.p.}(T) =$

$\; \; = \{\lambda \in \sigma(T) : \textrm{the operator}\; \lambda I - T \; \textrm{has the left inverse}\} $.

\begin{remark} \label{r1} The notations $\sigma_{a.p.}(T)$ and $\sigma_r(T)$ refer, of
course, to the approximate point spectrum and the residual spectrum of $T$,
respectively. But, because the corresponding definitions vary in the literature, we
prefer to avoid using this terminology.
\end{remark}

Following~\cite{EE} we consider the following spectra of $T$.

$\sigma_1(T) = \{\lambda \in \mathds{C}: \lambda I - T \not \in \Phi\}$ is the
\textit{semi-Fredholm} spectrum of $T$.

$\sigma_2(T) = \{\lambda \in \mathds{C}: \lambda I - T \not \in \Phi_+\}$.

$\sigma_3(T) = \{\lambda \in \mathds{C}: \lambda I - T \not \in \mathcal{F}\}$ is the
Fredholm spectrum of $T$.

$\sigma_4(T) = \{\lambda \in \mathds{C}: \lambda I - T \not \in \mathcal{W}\}$ is the
Weyl spectrum of $T$.

$\sigma_5(T) = \sigma(T)\setminus \{\zeta \in \mathds{C} :$ there is a component $C$ of
the set $\mathds{C} \setminus \sigma_1(T)$ such that $\zeta \in C$ and the intersection
of $C$ with the resolvent set of $T$ is not empty$\}$.

It is well known (see e.g.~\cite{EE}) that the sets $\sigma_i(T ), i \in [1, \ldots ,
5]$ are
nonempty closed subsets of $\sigma(T)$ and that $\sigma_i(T) \subseteq \sigma_j(T), 1
\leq i < j \leq 5$, where all the inclusions can be proper. Nevertheless all the
spectral radii $\rho_i(T ), i = 1, . . . , 5$ are equal to the same number,
$\rho_e(T)$, (see~\cite[Theorem I.4.10]{EE}) which is called the essential
spectral radius of $T$. It is also known (see~\cite{EE}) that the spectra $\sigma_i(T),
i = 1, \ldots , 4$ are invariant under compact perturbations, but $\sigma_5(T)$ in
general is not.

We will need two results on semi-Fredholm operators. The first of them is the following
well known lemma (see e.g.~\cite[Lemma 1]{Zi} or~\cite[Theorems 4.2.1, 4.2.2, and
4.4.1]{CPY})

\begin{lemma} \label{l1} Let $A, B \in \Phi$ and $f$ be a continuous map from $[0,1]$
into the Banach algebra $L(X)$ such that $f(0) = A, f(1) = B$, and $f([0,1]) \subset
\Phi$.
Then ind $A$ = ind $B$.

\end{lemma}

The second result was proved in~\cite[Theorem 1]{Sc}

\begin{theorem} \label{t22} Let $X$ be a Banach space and  $T \in L(X)$. Assume that
$\lambda \in \partial \sigma(T)$ and that the operator $\lambda I - T$ is
semi-Fredholm. Then $\lambda $ is a pole of the resolvent of $T$.
\end{theorem}

Theorems~\ref{t1} - ~\ref{t3} below are corollaries of more general results proved
in~\cite{Ki1} -~\cite{Ki3}.

\begin{theorem} \label{t1} Let $K$ be a compact Hausdorff space and $\varphi$ be a
homeomorphism of $K$ onto itself. Let $w \in C(K)$ and let
$$ (Tf)(k) = w(k)f(\varphi(k)), f \in C(K), k \in K. $$
Assume that
\begin{enumerate}[(a)]
  \item The set of all $\varphi$-periodic points is of first category in $K$.
  \item There is no open nonempty subset $O$ of $K$ such that the sets $\varphi^j(O),
      j \in \mathds{Z}$ are pairwise disjoint (where $\varphi^j$ means the $j$-th
      iteration of $\varphi$).

      Then $\sigma_1(T) = \sigma(T)$ is a rotation invariant subset of the complex
      plane.
  \item If additionally $K$ cannot be represented as union of two disjoint nonempty
      clopen $\varphi$-invariant subsets then $\sigma(T)$ is connected.
 \end{enumerate}

\end{theorem}

\begin{theorem} \label{t2} Let $A$ be a unital uniform Banach algebra and $U$ be an
automorphism of $A$. Let $w \in A$ and $T = wU$. Let $\varphi$ be the homeomorphism of $\mathfrak{M}_A$ generated by $U$. (Notice that $\varphi(\partial A) = \partial A$). The operator $T$ can be considered as a weighted composition operator on
$C(\mathfrak{M}_A)$ and on $C(\partial A)$. Then
\begin{enumerate}[(a)]
  \item $\sigma(T) = \sigma(T,C(\mathfrak{M}_A))$,
  \item $\sigma_{a.p.}(T) = \sigma_{a.p.}(T,C(\partial A))$.
\end{enumerate}
\end{theorem}

Recall the following definition.

\begin{definition} \label{d1} A Banach space $X$ is called a Banach ideal space (see
e.g.~\cite{KA}) if there is a measure space $(E, \mu)$ such that $X$ is an order ideal in the vector lattice $L^0(E, \mu)$ of all (classes of) $\mu$-measurable functions on $E$ and the norm on $X$ is a lattice norm compatible with the order on $X$, i.e. $X$ is a Banach space with norm $\|\cdot\|$ such that $x, y \in X, |x| \leq |y| \Rightarrow \|x\| \leq \|y\|$.

\end{definition}

\begin{theorem} \label{t3} Let $K$ be a compact Hausdorff space and $\mu$ be a finite
regular Borel probability measure on $K$. Let $\varphi$ be a measure preserving
homeomorphism of $K$ onto itself such that $\mu(\Pi) = 0$ where $\Pi$ is the set of all
$\varphi$--periodic points in $K$. Assume that
\begin{enumerate}
  \item $X$ is a Banach ideal space of $\mu$-measurable functions, and
  \item the ideal center $Z(X)$ is isomorphic to $L^\infty(K,\mu)$, and
  \item the composition operator $U$, $Ux = x \circ \varphi$, is bounded on $X$ and
      $\sigma(U) \subseteq \mathds{T}$.
\end{enumerate}
(In particular, conditions (1) - (3) above are satisfied if $X$ is an interpolation
space (see e.g.~\cite{BS}) between $L^\infty(K,\mu)$ and $L^1(K,\mu)$.)

Let $w \in L^\infty(K, \mu)$, and let $T$ be the weighted composition operator,
$$ Tx = w(x \circ \varphi), x \in X. $$
Then $\sigma_1(T) = \sigma(T)$ and $\sigma(T)$ is a rotation invariant subset of the
complex plane.

Moreover, if $\varphi$ is ergodic then the set $\sigma(T)$ is connected.

\end{theorem}

The next theorem provides a formula for the spectral radius of some weighted
composition operators (for the proof see~\cite{Ki5}).

\begin{theorem} \label{t4} Let $X$ be a Banach space and $\mathcal{A}$ be a closed
unital commutative subalgebra of $L(X)$. Assume that for every $a, b \in \mathcal{A}$
\begin{equation} \label{eq1} \|ab\| \leq C(\|a\|\|\hat{b}\|_\infty +
\|b\|\|\hat{a}\|_\infty),
\end{equation}
where $\hat{a}$ is the Gelfand transform of $a$ and $\|\hat{a}\|_\infty$ is the norm of
$\hat{a}$ in $C(\mathfrak{M}_{\mathcal{A}})$.

Let $U  \in L(X)$ be such that $\sigma(U) \subseteq \mathds{T}$ and $U\mathcal{A}U^{-1}
= \mathcal{A}$. Let $\varphi$ be the homeomorphism of $\partial \mathcal{A}$ generated
by the automorphism $a \rightarrow UaU^{-1}$ of $\mathcal{A}$. Finally, let $w \in
\mathcal{A}$ and $T = wU$. Then
\begin{equation} \label{eq2} \rho(T) = \max \limits_{\mu \in M_\varphi} \exp \int \ln
|\hat{w}| d\mu ,
\end{equation}
where $M_\varphi$ is the set of all $\varphi$-invariant regular probability Borel
measures on $\partial \mathcal{A}$.
\end{theorem}

In the sequel we will often use the following definition.

\begin{definition} \label{d7} Let $X$ be a Banach space and $T \in L(X)$. Let
$\varepsilon \in (0, 1)$ and $n \in \mathds{N}$. We define the operator $S_n(T,
\varepsilon)$ as
\begin{equation} \label{eq18}
S_n(T, \varepsilon) = \sum \limits_{j=0}^{2n} (1 - \varepsilon)^{|j-n|} T^j.
\end{equation}
\end{definition}

The next lemma follows from~(\ref{eq18}) by a direct computation.

\begin{lemma} \label{l3}

\begin{multline} \label{eq19}
(I - T)S_n(T, \varepsilon) = (1 - \varepsilon)^n I + \varepsilon \sum \limits_{j=1}^n
(1- \varepsilon)^{n-j} T^j \\
 - \varepsilon \sum \limits_{j=1}^n (1 - \varepsilon)^j T^{j+n} -(1 -
 \varepsilon)^{n+1} T^{2n + 1}.
\end{multline}
\end{lemma}

We will conclude this section with a simple but useful lemma (see~\cite[Lemma 3.6, p.
643]{Ki1})

\begin{lemma} \label{l4} Let $K$ be a compact Hausdorff space, $\varphi$ be a
homeomorphism of $K$ onto itself, and $w \in C(K)$. Let $T \in L(C(K))$ be defined as
$$ (Tf)(k) = w(k)f(\varphi(k)), f \in C(K), k \in K. $$
Let $\lambda \in \mathds{C}$, $\lambda \neq 0$. Then $\lambda \in \sigma_{a.p.}(T)$ if
and only if there is a $k \in K$ such that
\begin{equation} \label{eq20}
|w_n(k)| \geq |\lambda|^n \check{}\;  \text{and}\; |w_n(\varphi^{-n}(k))| \leq |\lambda|^n, n
\in \mathds{N}.
\end{equation}

\end{lemma}

\bigskip

\section{Weighted rotation-like operators and some properties of their spectra}

\begin{definition} \label{d2} Let $X$ be a Banach space and $U$ be an invertible
element of $L(X)$. We say that $U$ is a rotation-like operator if there is an $A \in
L(X)$, $A \neq 0$, and a $\gamma \in \mathds{T}$ such that $\gamma \neq 1$ and
$UAU^{-1} = \gamma A$.

\end{definition}

\begin{remark} \label{r2} Without the assumption that $\gamma \neq 1$
Definition~\ref{d2} becomes meaningless: every invertible operator on $X$ would be
"rotation-like".

\end{remark}

\begin{definition} \label{d3} Let $U$ be a rotation-like operator on $X$. Let
$$ \mathcal{A} = \{A \in L(X)\setminus \{0\}: \exists \gamma_A \in \mathds{T} \;
\text{such that} \; UAU^{-1} = \gamma_A A \} ,$$
$$ \Delta = \{\gamma_A : A \in \mathcal{A}\}. $$
It is obvious that $\mathcal{A}$ is a multiplicative unital semigroup in $L(X)$ and
that $\Delta$ is a unital semigroup of $\mathds{T}$. We denote the subgroup of
$\mathds{T}$ generated by $\Delta$ by $\Gamma$, i.e.
$$ \Gamma = \{\gamma, \bar{\gamma}: \gamma \in \Delta \}.$$
\end{definition}

\begin{definition} \label{d4} Let $U$ be a rotation-like operator.
Let $\mathcal{W} = \mathcal{A}^\prime$ be the commutant of $\mathcal{A}$ in $L(X)$ and
let $w \in \mathcal{W}$. We call the operator $T = wU$ a \textit{weighted
rotation-like} operator.
\end{definition}
\begin{theorem} \label{t5} Let $T = wU$ be a weighted rotation-like operator and
$\lambda \in \sigma_{a.p.}(T)$. Let $A \in \mathcal{A}$. Assume that the operator $A$
is invertible from the left.

Then, $\gamma_A^n \lambda \in \sigma_{a.p.}(T), n \in \mathds{Z}$.

Moreover, if $\lambda \in \sigma_2(T)$ then $\gamma_A^n \lambda \in \sigma_2(T), n \in
\mathds{Z}$.

\end{theorem}

\begin{proof}

 Let $\lambda \in \sigma_{a.p.}(T)$ and let $x_n \in X$, $\|x_n\| = 1$, $\lambda x_n -
 Tx_n \rightarrow 0$. Then $\lambda Ax_n - ATx_n \rightarrow 0$. But
$$ \lambda Ax_n - ATx_n = \lambda Ax_n - AwUx_n = \lambda Ax_n - wAUx_n = $$
$$ = \lambda Ax_n - wUU^{-1}AUx_n = \lambda Ax_n - \bar{\gamma}_A wUAx_n = $$
$$ = \bar{\gamma}_A(\gamma_A \lambda Ax_n - TAx_n).$$
Recalling that the operator $A$ is bounded from below we see that $\gamma_A \lambda \in \sigma_{a.p.}(T)$. Hence $\gamma_A^n \lambda \in \sigma_{a.p.}(T), n \in \mathds{N}$.  Now, whether $\gamma_A$ is a root of unity or not, the statement that $\gamma_A^n \lambda \in \sigma_{a.p.}(T), n \in \mathds{Z}$ becomes trivial, because in the latter case we have $\lambda \mathds{T} \subseteq
\sigma_{a.p.}(T)$.

Assume now that $\lambda \in \sigma_2(T)$. It is equivalent (see e.g.~\cite{BS}) to the
existence of a sequence $x_n \in X$ such that $\|x_n\| = 1$, $\lambda x_n - Tx_n
\rightarrow 0$, and the sequence $x_n$ is singular, i.e. it does not contain a norm
convergent subsequence. By the first part of the proof all we need is to prove that the
sequence $Ax_n$ is also singular. If not, we can assume without loss of generality that
$Ax_n \rightarrow y \in X$. Let $S$ be a left inverse of $A$. Then $x_n = SAx_n
\rightarrow Sy$, a contradiction.
\end{proof}

\begin{corollary} \label{c1} Let $T$ be a weighted rotation-like operator. Assume one
of the following conditions.
\begin{enumerate}
  \item Every operator from $\mathcal{A}$ has a left inverse and the group $\Gamma$
      is of infinite order.
  \item There is an $A \in \mathcal{A}$ such that $A$ has a left inverse and
      $\gamma_A$ is not a root of unity.
\end{enumerate}

Then the sets $\sigma(T)$, $\sigma_{a.p.}(T)$, and $\sigma_2(T)$ are rotation
invariant.

\end{corollary}

\begin{corollary} \label{c2} Let $T$ be a weighted rotation-like operator. Assume one
of the following conditions.
\begin{enumerate}
  \item Every operator from $\mathcal{A}$ is invertible and the group $\Gamma$ is of
      infinite order.
  \item There is an $A \in \mathcal{A}$ such that $A$ is invertible and $\gamma_A$ is
      not a root of unity.
\end{enumerate}

Then the spectrum, $\sigma(T)$, the essential spectra $\sigma_i(T), i = 1 \ldots 5$, as well as $\sigma_2(T^\prime)$, are rotation invariant.

\end{corollary}

\begin{proof} The set $\sigma_2(T^\prime)$ is rotation invariant in virtue of
Corollary~\ref{c1} and the fact that if $UAU^{-1} = \gamma_A A$ then $U^\prime A^\prime
(U^{-1})^\prime = \bar{\gamma}_A A^\prime$.

Next, the relations $\sigma_1(T) = \sigma_2(T) \cap \sigma_2(T^\prime)$ and
$\sigma_3(T) = \sigma_2(T) \cup \sigma_2(T^\prime)$ show that the sets $\sigma_1(T)$
and $\sigma_3(T)$ are rotation invariant.

To prove that $\sigma_4(T)$ is rotation invariant let $\lambda \in \sigma_4(T)
\setminus \sigma_3(T)$, i.e. the operator $\lambda I - T$ is Fredholm but its index is
not equal to $0$. Then $\sigma_3(T) \cap \lambda \mathds{T} = \emptyset$, i.e. for
every $\xi \in \lambda \mathds{T}$ the operator $\xi I - T$ is Fredholm. Because the
set of Fredholm operators is open in $L(X)$ and the index of a Fredholm operator is
stable under small norm perturbations (see e.g.~\cite{Kat}) we see that $\lambda
\mathds{T} \subseteq \sigma_4(T)$.

Finally, we can conclude that the set $\sigma_5(T)$ is rotation invariant based on its
definition and the fact that both $\sigma_1(T)$ and the resolvent set of $T$ are
rotation invariant.
\end{proof}

The  conditions of invertibility or one-sided invertibility we had to impose in the
previous results, are quite heavy and it is desirable to weaken them. That leads to the
following problem.

\begin{problem} \label{pr1} Let $T$ be a weighted rotation-like operator. Assume one of
the following conditions.
\begin{enumerate}
  \item The group $\Gamma$ is of infinite order and for every $A \in \mathcal{A}$ and
      every $n \in \mathds{N}$ we have $\ker{A^n} = 0$ (respectively, $\ker{A^n} =
      \ker{(A^\prime)^n} = 0$).

  \item For an $A \in \mathcal{A}$ such that $\gamma_A$ is not a root of unity and
      for every $n \in \mathds{N}$ we have $\ker{A^n} = 0$ (respectively, $\ker{A^n}
      = \ker{(A^\prime)^n} = 0$).
      \end{enumerate}

      \noindent Is it true that under these assumptions the statement of
      Theorem~\ref{t5} (respectively, Corollary~\ref{c2}) remains correct?

\end{problem}

We will now state and prove some partial results we obtained when trying to solve
Problem~\ref{pr1}.

\begin{lemma} \label{l2} Let $T = wU$ be a weighted rotation-like operator. Assume that
$A \in \mathcal{A}$ and $\gamma_A$ is not a root of unity. Assume also that there is a
sequence of polynomials $\{p_k\}$ such that
\begin{equation} \label{eq3} p_k(0) = 0
\end{equation}
 and
\begin{equation} \label{eq4} \|w - p_k(A)\| \mathop \rightarrow \limits_{k \to \infty}
0.
\end{equation}
Then the sets $\sigma(T)$ and $\sigma_{a.p.}(T)$ are rotation invariant.
\end{lemma}

\begin{proof} Let $\lambda \in \sigma_{a.p.}(T)$. Without loss of generality we can
assume that $\lambda \neq 0$. Let $x_n \in X$, $\|x_n\|=1$, and
\begin{equation} \label{eq5} T x_n - \lambda x_n \mathop \rightarrow \limits_{n \to
\infty} 0.
\end{equation}
The proof of Theorem~\ref{t5} shows that it is sufficient to prove that $ A x_n \mathop
\nrightarrow \limits_{n \to \infty} 0$. Assume to the contrary that $ A x_n \mathop
\rightarrow \limits_{n \to \infty} 0$. Then $ A Ux_n = \bar{\gamma}_A U A x_n \mathop
\rightarrow \limits_{n \to \infty} 0$ and conditions~(\ref{eq3}) and~(\ref{eq4})
guarantee that $ wUx_n \mathop \rightarrow \limits_{n \to \infty} 0$, in contradiction
with~(\ref{eq5}).
\end{proof}

\begin{theorem} \label{t6} Let $T = wU$ be a weighted rotation-like operator.
Assume the following conditions.
\begin{enumerate}[(a)]
  \item There is an $A \in \mathcal{A}$ such that $\gamma_A$ is not a root of unity.
  \item The weight $w$ belongs to the closure in the operator norm of the subalgebra
      generated by $A$ and the identical operator $I$ in $L(X)$.
  \item For any $n \in \mathds{N}$ we have $ \ker{A^n} = 0$.
  \item $\sigma(U) \subseteq \mathds{T}$.
\end{enumerate}
Then the sets $\sigma(T)$ and $\sigma_{a.p.}(T)$ are rotation invariant.
\end{theorem}

\begin{proof} I. Let us first consider the case when $w = p(A) = \sum_{j=0}^k a_jA^j$.
If $a_0 =0$ our statement follows from Lemma~\ref{l2}. Therefore, we can assume without
loss of generality that $a_0 = 1$. Let $\lambda \in \sigma_{a.p.}(T)$. Let $x_n \in X$,
$\|x_n\| = 1$, and $Tx_n - \lambda x_n \mathop \rightarrow \limits_{n \to \infty} 0$.
We claim that if $|\lambda| \neq 1$ then $\lambda \mathds{T} \subseteq \sigma(T)$.
Indeed, if $Ax_n \not \rightarrow 0$ then the inclusion $\lambda \mathds{T} \subseteq
\sigma(T)$ follows from the proof of Theorem~\ref{t5}. If, on the other hand, $Ax_n
\rightarrow 0$ then $Ux_n - \lambda x_n \rightarrow 0$, in contradiction with
$\sigma(U) \subseteq \mathds{T}$.

Thus, the problem is reduced to the following. Let $1 \in \sigma_{a.p.}(T)$ and let $1$
be an isolated point in the set $|\sigma(T)| = \{|\lambda| : \lambda \in \sigma(T)\}$.
Then we need to prove that $\mathds{T} \subseteq \sigma_{a.p.}(T)$.

Notice that the set $\sigma(T) \cap \mathds{T}$ is a clopen subset of $\sigma(T)$. Let
$\mathcal{P}$ be the corresponding spectral projection and $\mathcal{P}X$ be the
corresponding nontrivial spectral subspace of $T$ and $\sigma(T|\mathcal{P}X) =
\sigma(T) \cap \mathds{T}$ (see e.g.~\cite[p. 575]{DS}).

Next notice that the space $\mathcal{P}X$ is $A$-invariant. To prove it let $x \in
\mathcal{P}X$. Then $(wU)^n Ax = \lambda_A^n A(wU)^n x$ and, because
$\sigma(wU|\mathcal{P}X) \subseteq \mathds{T}$ we have
$$\limsup \limits_{n \to \infty} \|(wU)^n Ax\|^{1/n} \leq 1.$$
A similar reasoning shows that
$$\limsup \limits_{n \to \infty} \|(wU)^{-n} Ax\|^{1/n} \leq 1,$$
and therefore $Ax \in \mathcal{P}X$.

Let us denote by $\tilde{T}$ and $\tilde{A}$ the restrictions of, respectively, $T$ and
$A$ on $\mathcal{P}X$. We need to prove that $\sigma(\tilde{T}) = \mathds{T}$. If not,
then there is an open interval $I \subset \mathds{T}$ such that the resolvent
$\rho(\zeta, \tilde{T})$ is analytic in $(\mathds{C} \setminus \mathds{T}) \cup I$. The
formulas
\begin{equation*}\rho(\zeta,\tilde{T})=
\left\{
  \begin{array}{ll}
   \sum \limits_{n=0}^\infty -\zeta^n \tilde{T}^{-(n+1)}, & \hbox{if $|\zeta|<1$;} \\
   \sum \limits_{n=0}^\infty \zeta^{-(n+1)}\tilde{T}^n, & \hbox{if $|\zeta| > 1$,}
  \end{array}
\right.
\end{equation*}
and
$$ \tilde{T}^n \tilde{A} = \lambda_A^n \tilde{A} \tilde{T}^n, n \in \mathds{Z},$$
show that for any $N \in \mathds{N}$ and for any $x \in \mathcal{P}X$ the vector
function $\rho(\zeta, \tilde{T})\tilde{A}^N x$ is analytic in $(\mathds{C} \setminus
\mathds{T}) \cup \bigcup \limits_{j=0}^N \lambda_A^{-j}I$. Because $\lambda_A$ is not a
root of unity, for any large sufficient $N$ the function $\rho(\zeta,
\tilde{T})\tilde{A}^N x$ is analytic in $\mathds{C}$ and therefore $\tilde{A}^N = 0$ in
contradiction with condition $(c)$.

II. Consider the general case. Let $x_n \in X$, $\|x_n\|=1$, and  $wUx_n - \lambda x_n
\rightarrow 0$. If $A x_n \nrightarrow 0$ then $\lambda \mathds{T} \subseteq
\sigma(wU)$; therefore assume that $A x_n \rightarrow 0$. For any $N \in \mathds{N}$ we
can find a polynomial $p_N$, $p_N(x) = \sum \limits_{j=0}^{m(N)} a_{j,N} x^j$, such
that $\|p_N(A)U - wU\| < 1/N$. Then $\limsup \limits_{n \to \infty} \|a_{0,N} Ux_n -
\lambda  x_n\| \leq 1/N $, and therefore the sequence $|a_{0,N}|, N \in \mathds{N},$ is
bounded. Taking, if necessary, a subsequence of this sequence we can assume that $\lim
\limits_{N \to \infty} a_{0,N} = a_0 \in \mathds{C}$, and therefore $\lambda \in a_0
\sigma(U) = a_0 \mathds{T}$. The proof can now be finished as in part I.
\end{proof}

 The proof of the following theorem is similar to that of Theorem~\ref{t6} and
 therefore we omit it.

\begin{theorem} \label{t7} Let $T = wU$ be a weighted rotation-like operator.
Assume the following conditions.
\begin{enumerate}[(a)]
  \item The group $\Gamma$ is of infinite order.
  \item The operators from $\mathcal{A}$ commute.
  \item The weight $w$ is invertible in $L(X)$ and belongs to the closure in the
      operator norm of the subalgebra generated by $\mathcal{A}$ and the identical
      operator $I$ in $L(X)$.
  \item For any $A \in \mathcal{A}$ we have $ \ker{A} = 0$.
  \item $\sigma(U) \subseteq \mathds{T}$.
\end{enumerate}
Then the set $\sigma(T)$ is rotation invariant.

\end{theorem}

We can relax conditions of Theorem~\ref{t6} but at the price of being able to prove
only a considerably weaker result.

\begin{theorem} \label{t8} Assume that

\begin{enumerate} [(A)]
  \item $T$ is an invertible weighted rotation-like operator.
  \item For any $A \in \mathcal{A}$ and for any $n \in \mathds{N}$ we have $A^n \neq
      0$.
  \item The group $\Gamma$ is of infinite order.
\end{enumerate}

Then there is a real positive number $t$ such that
$$ t\mathds{T} \subseteq \sigma(T). $$
\end{theorem}

\begin{proof}. First notice that for any $A \in \mathcal{A}$ we have $(wU)A(wU)^{-1} =
wUAU^{-1}W^{-1} = w\gamma_A A w^{-1} = \gamma_A A$. Therefore, we can assume without
loss of generality that $w=I$. Assume, contrary to our statement, that $\sigma(U)$ does
not contain any circle centered at $0$.
Let $R = \rho(U)$ and $r = 1/\rho(U^{-1})$. There are an $m \in \mathds{N}$ and numbers
$\theta_j \in [0, 2\pi), j=1, \ldots, m$ such that $0 \leq \theta_j$, $\theta_j +
2\pi/m < 2\pi$, and for any $j \in [0 ; m-1]$ we have
$$\Delta_j =  \{ue^{i\theta} : r_j \leq u \leq r_{j+1},  \theta_{j+1} \leq \theta \leq
\theta_j + 2\pi/m\} \subset \mathds{C} \setminus \sigma(U),  \eqno{(5)}$$
where $r_j = r +j(R-r)/m$.

 Condition (C) of the theorem guarantees that there are an $A \in \mathcal{A}$ and an
 $N \in \mathds{N}$ such that for any $j \in [0 : m-1]$
\begin{equation} \label{eq6} \{ue^{i\theta} : r_j \leq u \leq r_{j+1}, \; 0 \leq \theta
\leq 2\pi\} \subseteq \bigcup \limits_{n=0}^N \gamma_A^n \Delta_j.
\end{equation}
Fix an arbitrary $x \in X$. The vector valued function $\mathcal{R}(\lambda) = (\lambda
I - U)^{-1}x$ is analytic in the region $\{\lambda \in \mathds{C}: |\lambda| > R$\} and
$(5)$ guarantees that it can be analytically extended on some open neighborhood of
$\Delta_{m-1}$. From $(1)$ easily follows that
\begin{equation} \label{eq7} A^n(\lambda I - U)^{-1}x =\gamma_A(\lambda \gamma_A I -
U)^{-1}A^n x.
\end{equation}
Combining~(\ref{eq6}) and~(\ref{eq7}) we see that the function $(\lambda I - U)M^N x$
is analytic in the region
$\{\lambda \in \mathds{C}: |\lambda| > r_{m-1}\}$. Repeating this argument we come to
the conclusion that the function $(\lambda I - U)A^{mN} x$ is analytic in $\mathds{C}$
and therefore identically zero.
Thus, $A^{mN} = 0$, a contradiction.
\end{proof}

Next we will discuss some conditions of absence (or presence) of circular gaps in the
spectrum of $T$.

\begin{theorem} \label{t9} Let $U$ be a rotation-like operator such that $\sigma(U)
\subseteq \mathds{T}$.

Let $w$ be an invertible weight such that $w \in \mathcal{A}^{\prime \prime}$ where
$\mathcal{A}^{\prime \prime}$ is the double commutant of $\mathcal{A}$ and let $T =
wU$.

Assume that there is a circular gap in $\sigma(T)$, i.e. there is a positive real
number $r$ such that $\sigma(T)$ is the union of two nonempty sets, $\sigma_1$ and
$\sigma_2$, such that $\sigma_1 \subset \{z \in \mathds{C} : |z| < r\}$ and
$\sigma_2 \subset \{z \in \mathds{C} : |z| > r\}$. Let $X_1$ and $X_2$ be the
corresponding spectral subspaces of $T$ and $P_1$, $P_2$ - the corresponding spectral
projections.

\noindent Then the projections $P_1$, $P_2$ commute with $U$ and moreover,
$P_1, P_2 \in \mathcal{A} \cap \mathcal{A}^{\prime \prime}$. \footnote{We do not assume
that $\mathcal{A}$ is a \textit{commutative} semigroup.}
\end{theorem}

\begin{proof} First notice that because $U\mathcal{A}U^{-1} = \mathcal{A}$ we have

\begin{equation} \label{eq8} U \mathcal{A}^{\prime \prime} U^{-1} = \mathcal{A}^{\prime
\prime}.
\end{equation}

Next, $(wU)^n = w_nU^n$, where $w_n = w(UwU^{-1})\ldots (U^{n-1}wU^{-(n-1)})$. The
conditions of the proposition together with equality~(\ref{eq8}) guarantee that the
operators $U^jwU^{-j}, j=0,1, \ldots,$ pairwise commute.

Therefore, if $x \in X_1$ and $n \in \mathds{N}$, then  it is immediate to see that
\begin{equation} \label{eq9} (wU)^nUx =U^n w^{-1} U^{-n} (wU)^{n+1}x.
\end{equation}
From~(\ref{eq9}) and from the fact that $x \in X_1$ and $\sigma(U) \subseteq
\mathds{T}$ we obtain that
$\lim \limits_{n \to \infty} \|(wU)^n Ux\|^{1/n} < r$, hence $Ux \in X_1$ and $UX_1
\subseteq X_1$.

Similarly, for any $x \in X_1$ and any $n \in \mathds{N}$ we have
$$ (wU)^n U^{-1}x = U^{n-1}wU^{-(n-1)}(wU)^{n-1}x, $$
and therefore $U^{-1}X_1 \subseteq X_1$.

By considering the operator $(WU)^{-1}$ we obtain in the same way that $UX_2 \subseteq
X_2$ and $U^{-1}X_2 \subseteq X_2$.

Hence, $UX_i = X_i, i=1,2$ and therefore $U$ commutes with projections $P_i, i= 1,2$.

Next, let $A \in \mathcal{A}$ and $x \in X_1$. Then $(wU)^nMx = \lambda_A^n A(wU)^n x$
whence $Ax \in X_1$ and $AX_1 \subseteq X_1$. Similarly we obtain that $AX_2 \subseteq
X_2$. Therefore $A$ commutes with the spectral projections $P_i, i= 1,2$.

 Finally, Let $B \in \mathfrak{M}^{\prime \prime}$. Because the projections $P_i, i =
 1,2$ commute with $\mathcal{A}$ they commute with $B$.
\end{proof}

\begin{corollary} \label{c3} Assume conditions of Theorem~\ref{t9}. Assume also that
the either $\mathcal{A}$ or $\mathcal{A}^{\prime \prime}$ does not contain any
idempotent $P$ such that $UPU^{-1} = P$.

Then the set $\{|\lambda|: \lambda \in \sigma(T)\}$ is connected.

If, additionally, the set $\sigma(T)$ is rotation invariant then it is either an
annulus or a circle centered at $0$.

\end{corollary}

We finish this section with the discussion of the following question: is it possible in
some cases to extend the result of Theorem~\ref{t4} without assuming
condition~(\ref{eq1}) in the statement of this theorem? At the present we have only a
very limited result related to this problem.

\begin{theorem} \label{t10} Let $U$ be a rotation-like operator such that $\sigma(U)
\subseteq \mathds{T}$ and let $A \in \mathcal{A}$ be such that  $UAU^{-1} = \gamma_A
A$, where $\gamma_A$ is not a root of unity. Assume that $w = p(A)$ where $p$ is a
polynomial and that $w$ is invertible in $L(X)$. Assume also that there are sequences
$\{\gamma_m ; m \in \mathds{N}\}$ and $n_m \in \mathds{N}$ such that $n_m \uparrow
\infty$, $\gamma_m$ is a primitive $n_m^{th}$ root of unity, and for any positive real
number $C$ we have
\begin{equation} \label{eq10} \lim \limits_{m\to \infty} |\gamma_m - \gamma_A|C^{n_m}
=0.
\end{equation}
Then
$$ \rho(wU) = \max \limits_{\mu \in \mathcal{M}_\varphi} \exp \int \ln |\hat{w}| d\mu$$
where $\hat{w}$ is the Gelfand transform of $w$ considered as an element of the
commutative Banach algebra $\mathcal{B}$ which is the operator norm closure of the
algebra generated by $A$ and $I$, $\varphi$ is the homeomorphism of the Shilov boundary
$\partial \mathcal{B}$ generated by the automorphism $b \rightarrow UbU^{-1}, b \in
\mathcal{B}$, and $\mathcal{M}_\varphi$ is the set of all $\varphi$-invariant regular
Borel probability measures on $\partial \mathcal{B}$.
\end{theorem}

\begin{proof} Let $j$ be the degree of $p$ and let $c_1, \ldots, c_j$ be its roots in
$\mathds{C}$. Then
$p(A) = C(c_1I - A) \ldots (c_j I -A)$. Without loss of generality we can assume that
$C = 1$. Next,
\begin{equation} \label{eq11} (wU)^{n_m} = p(A)p(\gamma_A A) \ldots p(\gamma_A^{n_m
-1}A^{n_m -1}) U^{n_m}.
\end{equation}
If in the right part of~(\ref{eq11}) we substitute $\gamma_m$ for $\gamma_A$ the right
part becomes
\begin{equation} \label{eq12} (c_1^{n_m} - A^{n_m})\ldots (c_j^{n_m} - A^{n_m}).
\end{equation}
Condition~(\ref{eq10}) together with the condition $\sigma(U) \subseteq \mathds{T}$
guarantee that
\begin{equation} \label{eq13} \lim \limits_{m \to \infty} \|(wU)^{n_m}\|^{1/n_m}
 = \lim \limits_{m \to \infty} \| (c_1^{n_m} - A^{n_m})\ldots (c_j^{n_m} -
 A^{n_m})\|^{1/n_m} \end{equation}
 and
 \begin{equation} \label{eq14} \lim \limits_{m \to \infty}
 \|\hat{w}_{n_m}\|_\infty^{1/n_m}
 = \lim \limits_{m \to \infty} \| (c_1^{n_m} - \hat{A}^{n_m})\ldots (c_j^{n_m} -
 \hat{A}^{n_m})\|_\infty^{1/n_m}.
 \end{equation}

 Because $w$ is invertible and $\gamma_A$ is not a root of unity we see that $\sigma(A)
 = \sigma(\hat{A})$ does not intersect the circles centered at $0$  with the  radii
 $|c_1|, \ldots, |c_j|$. We have to consider three cases.

 $(1)$. $\rho(A) < \min \limits_{k=1, \ldots,j} |c_k|$. Then
 $$ (c_1^{n_m} - A^{n_m})\ldots (c_j^{n_m} - A^{n_m}) = \prod \limits_{k=1}^j c_k^{n_m}
 + R_m, $$
 where $\|R_m\| \leq d^{n_m}\prod \limits_{k=1}^j |c_k|^{n_m}$ for some constant $d$,
 $0 < d <1$. Therefore,
 \begin{multline} \label{eq15} \lim \limits_{m \to \infty} \| (c_1^{n_m} -
 A^{n_m})\ldots (c_j^{n_m} - A^{n_m})\|^{1/n_m} = \\
  = \lim \limits_{m \to \infty} \| (c_1^{n_m} - \hat{A}^{n_m})\ldots (c_j^{n_m} -
  \hat{A}^{n_m})\|_\infty^{1/n_m} = \prod \limits_{k=1}^j |c_k|.
 \end{multline}

$(2)$. $\rho(A) > \max \limits_{k = 1, \ldots, j} |c_k|$. Then, similarly to case $(1)$
we get
\begin{multline} \label{eq16} \lim \limits_{m \to \infty} \| (c_1^{n_m} -
A^{n_m})\ldots (c_j^{n_m} - A^{n_m})\|^{1/n_m} = \\
 = \lim \limits_{m \to \infty} \| (c_1^{n_m} - \hat{A}^{n_m})\ldots (c_j^{n_m} -
 \hat{A}^{n_m})\|_\infty^{1/n_m} = \rho(A)^j.
 \end{multline}
$(3)$. Finally, if we assume that there is a natural $l$, $1 \leq l < j$, such that
$|c_k| < \rho(A), 1 \leq k \leq l$, and $|c_k| > \rho(A), l < k \leq j$, then
\begin{multline} \label{eq17} \lim \limits_{m \to \infty} \| (c_1^{n_m} -
A^{n_m})\ldots (c_j^{n_m} - A^{n_m})\|^{1/n_m} = \\
 = \lim \limits_{m \to \infty} \| (c_1^{n_m} - \hat{A}^{n_m})\ldots (c_j^{n_m} -
 \hat{A}^{n_m})\|_\infty^{1/n_m} =\rho(A)^l \prod \limits_{k=l+1}^j |c_k|.
 \end{multline}

The statement of the theorem follows from~(\ref{eq13}) -~(\ref{eq17}).
\end{proof}

\begin{problem} \label{pr2} Will the statement of Theorem~\ref{t10} remain true if we
assume only that $\gamma_A$ is not a root of unity and $w \in \mathcal{B}$?
\end{problem}

\bigskip

\section{\centerline{Weighted rotation operators on Banach ideal spaces}}

In this subsection we will complement Theorem~\ref{t3} by some results concerning the
spectral radius of weighted composition operators on Banach ideal spaces. Then we will
consider in more details the case of weighted rotation-like operators acting on Banach
ideals of $L^0(G,m)$ where $G$ is a compact abelian group and $m$ is the Haar measure.

In what follows we assume notations and conditions from the statement of
Theorem~\ref{t3}.

We will also assume temporarily that the homeomorphism $\varphi$ is uniquely ergodic,
i.e. $\mu$ is the unique $\varphi$-invariant regular Borel probability measure on
$K$.

Let us agree (as it is customary) that if $x \in L^0(K,\mu)$ we will say that $x \in
C(K)$, or that $x$ is semicontinuous, et cetera, instead of stating more rigorously
that $x$, considered as a class of $\mu$-a.e. coinciding $\mu$-measurable functions on
$K$, has a representative that is continuous (respectively, semicontinuous, et cetera)
on $K$.

By Theorem~\ref{t4}
$$ \rho(T) = \rho(|w|) = \max \limits_{\nu \in \mathfrak{M}_\psi} \exp \int \ln
|\hat{w}| d\nu $$
where $\psi$ is the homeomorphism of the space of maximal ideals, $\mathfrak{K}$, of
the algebra $L^\infty(K,\mu)$ induced by the isomorphism $f \rightarrow UfU^{-1}$ of
$L^\infty(K,\mu) \simeq C(\mathfrak{K})$ (considered as a closed subalgebra of $L(X)$),
$\mathfrak{M}_\psi$ is the set of all $\psi$-invariant regular Borel probability
measures on $\mathfrak{K}$, and $\hat{w}$ is the Gelfand transform of $w$.

It follows that
$$ \varrho(T) = \exp \int \ln |w| d\mu \leq \rho(T). $$
On the other hand, applying again Theorem~\ref{t4} we see that if $|w| \in C(K)$ then
\begin{equation}\label{eq34}
   \varrho(T) = \rho(T).
\end{equation}
As the next example shows equality~(\ref{eq34}) becomes in general false if we assume
only that $w \in L^\infty(K, \mu)$.

\begin{example} \label{e1}  Let $m$ be the normalised Lebesgue measure on $\mathds{T}$.
Let $\varphi : \mathds{T} \rightarrow \mathds{T}$ be defined as $\varphi(t) = \alpha
t$, where $|\alpha| = 1$ and $\alpha$ is not a root of unity. Let $E$ be a nowhere
dense closed subset of $\mathds{T}$ such that $m(E) > 0$. Notice that for any $n \in
\mathds{N}$ we have $m(\mathds{T} \setminus \bigcup \limits_{i=-n}^n \varphi^i(E)) >
0$. Let $\mathfrak{K}$ be the Gelfand compact of the algebra $L^\infty(\mathds{T},m)$.
For any $f \in L^\infty(\mathds{T},m)$ let $\hat{f} \in C(\mathfrak{K})$ be the Gelfand
transform of $f$. Let $\hat{E}$ be the support of the function $\hat{\chi}_E$ in
$\mathfrak{K}$. Then $\hat{E}$ is a clopen subset of $\mathfrak{K}$ and $F =
\mathfrak{K} \setminus \bigcup \limits_{i=-\infty}^\infty \varphi^i(\hat{E}) \neq
\emptyset$. Let $\hat{w} \in C(\mathfrak{K})$ be such that $\hat{w} \equiv 2$ on $F$
and $\hat{w} \equiv 1$ on $\hat{E}$. Let $\psi$ be the homeomorphism of $\mathfrak{K}$
onto itself generated by the isomorphism $f \rightarrow f \circ \varphi$ of
$L^\infty(\mathds{T},m)$. Notice that $\psi(F) = F$ and therefore there exists a
$\psi$-invariant regular probability Borel measure $\nu$ on $\mathfrak{K}$ such that
$\exp{\int{ \ln{\hat{w}} d\nu}} = 2$. On the other hand, if $\hat{m}$ is the functional
on $C(\mathfrak{K})$ such that $\hat{m}(\hat{f}) = \int f dm$ then it is immediate that
$\hat{m}$ is a $\hat{\varphi}$-invariant regular probability Borel measure on
$\mathfrak{K}$ and that $ \exp \int \ln \hat{w} d\hat{m} < 2$. Because $m$ is the only
$\varphi$-invariant probability Borel measure on $\mathds{T}$ for any $\tilde{w} \in w$
we have
$$ \varrho(T) =  \exp \int \ln \hat{w} d\hat{m} < 2 = \rho(T). $$
$\bigstar$
\end{example}

Nevertheless, the condition $|w| \in C(K)$ is not necessary for $\varrho(T) = \rho(T)$.

\begin{theorem} \label{t11} Assume conditions of Theorem~\ref{t3}. Assume also that the
map $\varphi$ is uniquely ergodic. Let $|w|$ be an upper semicontinuous function on
$K$. Then
$$ \varrho(T) = \rho(T). $$
\end{theorem}

\begin{proof} Being an upper semicontinuous function $|w|$ is the lower envelope of the
set
$\{f \in C(K) : |w| \leq f\}$ (see~\cite[p.146]{Bo1}), i.e.
$$ \forall k \in K, |w(k)| = \inf\{f(k) : f \in C(K) : |w| \leq f\}.$$
It follows that $|w| = \inf \limits_{L^0(K,\mu)} \{f \in C(K) : |w| \leq f\}$.
Consider first the case when $\int \ln |w| d\mu > -\infty$ i.e. $\ln |w| \in L^1(K,
\mu)$ For any $f \in C(K)$ such that $|w| \leq f$ let $T_f = fU$. Then
$$\rho(T) \leq \inf\{\rho(T_f) : f \in C(K), |w| \leq f\} = \inf\{\exp \int \ln f d\mu
: f \in C(K), |w| \leq f\}. $$
But the functional $F(x) = \int x d\mu $ is order continuous on $L^1(K,\mu)$, hence
$$\inf\{\exp \int \ln f d\mu : f \in C(K), |w| \leq f\} = \exp \int \ln |w| d\mu =
\varrho(T).$$
Assume now that $\int \ln |w| d\mu = -\infty$ and for each $n \in \mathds{N}$ consider
$g_n = |w| + 1/n$ and $T_n = g_n U$. Then by the previous part of the proof we have
$$ \rho(T) \leq \rho(T_n) = \exp \int \ln(|w| +1/n)d\mu \mathop \rightarrow \limits_{n
\to \infty} 0. $$
\end{proof}

\begin{corollary} \label{c4} Assume conditions of Theorem~\ref{t11}. Then
\begin{enumerate}
  \item If $|w|$ is lower semicontinuous and invertible in $L^\infty(K,\mu)$ then
  $$ \rho(T^{-1}) = \exp \int - \ln |w| d\mu. $$
  \item If $|w|$ is $\mu$-Riemann integrable on $K$  \footnote{The definition of
      Riemann integrable function on a compact topological space endowed with a Borel
      measure can be found in~\cite[p.130]{Bo2}.} then
  $$ \rho(T) = \exp \int \ln |w| d\mu. $$
  \item If $|w|$ is $\mu$-Riemann integrable on $K$ and invertible in
      $L^\infty(K,\mu)$ then
      $$ \sigma(T) = \rho(T)\mathds{T}. $$
\end{enumerate}
\end{corollary}

\begin{proof} $(1)$ is trivial because if $|w|$ is invertible and lower semicontinuous
then $1/|w|$ is upper semicontinuous.

$(2)$ and $(3)$ follow from the fact that (see~\cite{Bo2}) if $|w|$ is $\mu$-Riemann
integrable then $\mu$-a.e. $|w| = g = h$ where $g$ (respectively, $h$)  is an upper
semicontinuous (respectively, lower semicontinuous) function.
\end{proof}

Recall that a topological group $G$ is called \textit{monothetic} if there is $g \in G$
such that the subgroup $\{g^n : n \in \mathds{Z}\}$ is dense in $G$. Clearly, every
monothetic group is abelian.

\begin{corollary} \label{c5} Let $G$ be a compact monothetic Hausdorff \footnote{The
condition that $G$ is Hausdorff is often included in the definition of topological
group.} topological group and let $h \in G$ be such that $cl\{h^n : n \in \mathds{Z}\}
= G$. Let $m$ be the Haar measure on $G$.
Let $\varphi(g) = hg, g \in G$ and let $w \in L^\infty(G,m)$.

Let $X$ be a Banach ideal in $L^0(G,m)$ such that the ideal center $Z(X) \cong
L^\infty(G,m)$, the composition operator $U, Ux = x \circ \varphi$ is defined and
bounded on $X$, and $\sigma(U) \subseteq \mathds{T}$. Finally let $T = wU \in L(X)$.

Then
\begin{enumerate}
  \item $\sigma(T) = \sigma_1(T)$ is a rotation invariant connected subset of
      $\mathds{C}$.
  \item If $|w|$ is upper semicontinuous then
  $$ \rho(T) = \exp \int \ln |w| dm. $$
  \item If $|w|$ is $m$-Riemann integrable then
  $$ \sigma(T) = \rho(T)\mathds{T}. $$
\end{enumerate}
\end{corollary}

We can prove a result similar to Theorem~\ref{t11} but not involving the condition that
$\varphi$ is uniquely ergodic. Let $(K, \mu)$ be a compact Hausdorff space with a
probability Borel measure $\mu$. Let $\varphi$ be a $\mu$-preserving homeomorphism of
$K$ onto itself and let $|w|$ be an upper semicontinuous function on $K$. Because $|w|$
is a bounded Borel function on $K$ the following expression is well defined
\begin{equation}\label{eq35}
  \inf \limits_{\nu \in M_\varphi} \exp \int \ln |w| d\nu,
\end{equation}
where $M_\varphi$ is the set of all $\varphi$-invariant regular Borel probability
measures on $K$.

\begin{remark} \label{r3}  We need to emphasize that in~(\ref{eq35}) $|w|$ is
considered as an individual function, not as an element of $L^0(K,\mu)$. Indeed,
changing values of $|w|$ on a Borel set $E$ such that $\mu(E) = 0$ may change the value
of the expression in~(\ref{eq35}).
\end{remark}

The proof of the next theorem goes along the same lines as that of Theorem~\ref{t11}
and we omit it.

\begin{theorem} \label{t12} Assume conditions of Theorem~\ref{t3}. Assume additionally
that $|w|$ coincides $\mu$-a.e. with an upper semicontinuous function $\tilde{w}$. Then

$$\rho(T) = \inf \limits_{\nu \in M_\varphi} \exp \int \ln |\tilde{w}| d\nu.$$

\end{theorem}

\begin{problem} \label{pr3} Assume conditions of Theorem~\ref{t3}. Assume also that the
map $\varphi$ is uniquely ergodic and that $\rho(T) = \exp \int \ln |w| d\mu$. Is it
true that $|w|$ coincides $\mu$-a.e. with an upper semicontinuous function?
\end{problem}

We finish this section with the following variant of Theorem~\ref{t4} which we will
need in the sequel.

\begin{theorem} \label{t13} Let $G$ be a compact abelian group, $h \in G$, and $w \in
C(G)$. Let $(Tf)(g) = w(g)f(hg), g \in G, f \in C(G)$ and for any $t \in G$ let
$U_tf(g) = f(tg), g \in G, f \in C(G)$. Let $H = cl\{h^n : n \in \mathds{Z}$ be the
closed subgroup of $G$ generated by $h$. Let $m_H$ be the Haar measure on $H$  Then
$$ \rho(T) = \max \limits_{t \in G} \exp \int \ln |w| d(U_t^\prime m_H) . \eqno{(21)}
$$
\end{theorem}

\begin{proof} Assume first that $w$ is invertible in $C(G)$. For any $n \in \mathds{N}$
there is $g_n \in G$ such that
$$\|T^n\| = |w(g_n)w(hg_n) \cdots w(h^{n-1}g_n)| = \exp \int \ln |w| d\mu_n,$$
where $\mu_n = \frac{1}{n}\sum \limits_{i=0}^{n-1}(U_h^i)^\prime \delta_{g_n}$. Notice
that
$\mu_n = U_g^\prime \nu_n$ where $\nu_n = \frac{1}{n}\sum \limits_{i=0}^{n-1}
\delta_{h^i}$.
It is immediate to see that $\nu_n \rightarrow m_H$ where convergence is in
weak$^\star$ topology on $C(G)$. Let $g$ be a limit point of the sequence $g_n$ in $G$.
Then for any $f \in C(G)$ we have $\|U_{g_n}f - U_gf\|_{C(G)} \rightarrow 0$ and $(21)$
follows.

If $w$ is not invertible in $C(G)$ we will consider operators $T_m$, $(T_m f)(g) = (|w|
+1/m)f(hg)$ and then apply the first part of the proof and Lebesgue's Dominated
Convergence Theorem.

\end{proof}

\bigskip

\section{General properties of spectra of weighted rotation \\ operators
 on spaces of analytic functions}

\begin{center}
  The goal of this section is to describe some general properties of spectra of
  weighted rotations operators in  Banach spaces of analytic functions. These
  properties will be helpful later when we discuss some concrete Banach spaces of
  analytic functions.
\end{center}

\begin{theorem} \label{t19} Let $\Omega$ be an open connected subset of $\mathds{C}^n$
and $X$ be a Banach space such that $X \subset H(\Omega)$. Assume that $\dim X =
\infty$. Let $\varphi$ be an analytic automorphism of $\Omega$ and $w \in H(\Omega)$,
$w \not \equiv 0$. Assume that the operators $Ux = x \circ \varphi$ and $T = wU$ are
defined and bounded on $X$. Finally, assume that there are a compact subset $K$ of
$\Omega$ such that $\varphi(K) = K$, and a $\varphi$-invariant regular Borel
probability measure $\mu$ on $K$ such that for any $f \in H(\Omega)$ we have
\begin{equation} \label{eq21} \int \ln |f(\omega)|d\mu = - \infty \Rightarrow f \equiv
0.
\end{equation}
Then
\begin{enumerate}
  \item $\{|\lambda| : \lambda \in \sigma_p(T)\}$ is either empty or the singleton
      $\{s\}$, where $s = \exp \int \ln |w| d\mu$.
  \item $\sigma_{a.p.}(T) \setminus \sigma_1(T) \subset s\mathds{T}$.
  \item If the set $\{|\lambda| : \lambda \in \sigma_{a.p.}(T)\}$ is connected and
      there are no isolated points of $\sigma(T)$ on the circle $\rho(T)\mathds{T}$,
      then $\sigma_1(T) = \sigma_{a.p.}(T)$.
\end{enumerate}
\end{theorem}

\begin{proof} (1) Let  $x \in X$, $x \neq 0$, $\lambda \in \mathds{C}$ and $Tx =
\lambda x$. We can assume that $\lambda \neq 0$. Indeed, otherwise $w \equiv 0$. Then
\begin{equation} \label{eq22} \int \ln|w|d\mu + \int \ln |x \circ \varphi|d\mu = \ln
|\lambda| + \int \ln |x| d\mu .
\end{equation}
Because $\mu$ is $\varphi$-invariant, it follows from~(\ref{eq21}) and~(\ref{eq22})
that
\begin{equation} \label{eq28} |\lambda| = \exp \int \ln |w| d\mu.
\end{equation}
(2) Follows from (1) and the fact that $\sigma_{a.p.}(T) \setminus \sigma_1(T) \subset
\sigma_p(T)$.

\noindent (3) If $\{|\lambda| : \lambda \in \sigma_p(T)\}$ is a closed interval
$[a,b]$, $0 \leq a < b$, then the statement follows from (1) and the fact that
$Int_{\mathds{C}} (\sigma_{a.p.}(T) \setminus \sigma_1(T)) \neq \emptyset $.

Otherwise $|\lambda| = \rho(T)$ and our statement follows from the fact that
$\sigma(T)$ does not have isolated points and Theorem~\ref{t22}.
\end{proof}

\begin{problem} \label{pr8} Is it possible to dispense with the condition that
$|\sigma_{a.p.}(T)|$ is connected  in part (3) of the statement of Theorem~\ref{t19}?
\end{problem}

There is one special but important case when the answer to Problem~\ref{pr8} is
positive.

\begin{theorem} \label{t25} Assume conditions of Theorem~\ref{t19}. Assume additionally
that

(a) For every $m \in \mathds{N}$ the dynamical system $(K,\varphi^m)$ is minimal, i.e.
for every $k \in K$ the set $\{\varphi^{mn}(k) : n \in \mathds{N}\}$ is dense in $K$.

(b) For any open nonempty subset $O$ of $K$ we have

  \begin{equation}\label{eq23}
   \mu(O) > 0.
\end{equation}

 (c) The operator $U$ is a rotation-like operator and for any $A \in \mathcal{A}$ (see
 Definition~\ref{d3}), $\ker A = 0$.

 \noindent Then
\begin{enumerate}
  \item For any $\lambda \in \sigma_p(T)$, $\dim \ker (\lambda I - T) = 1$.
  \item $\sigma_1(T) = \sigma_{a.p.}(T)$.
\end{enumerate}
\end{theorem}

\begin{proof} (1) It follows immediately from~(\ref{eq21}), ~(\ref{eq23}), and Baire
category theorem that there is $\bar{k} \in K$ such that for any $n \in \mathds{N}$ we
have $w_n(\bar{k}) \neq 0$. Let $\lambda \in \sigma_p(T)$ and $x \in X$, $x \neq 0$, be
such that $Tx=\lambda x$. Notice that $x(\bar{k}) \neq 0$. Indeed, otherwise for any $n
\in \mathds{N}$ we have $x(\varphi^n(\bar{k}))= \frac{\lambda^n
x(\bar{k})}{w_n(\bar{k})} =0$, and it follows from the minimality of the system $(K,
\varphi)$ and from~(\ref{eq21})  that $x=0$. Now, if $y \in X \setminus \{0\}$ and $Ty
= \lambda y$ then for some $c \in \mathds{C} \setminus\{0\}$ we have $(y - sx)(\bar{k})
=0$ and therefore $y=sx$.

\noindent (2) Assume to the contrary that there is a $\lambda \in \sigma_{a.p.}(T)
\setminus \sigma_1(T)$. Then $\lambda \in \sigma_p(T)$. Let $x \in X \setminus \{0\}$
be such that $Tx = \lambda x$, and let $A \in \mathcal{A}$ be such that $UAU^{-1} =
\gamma A$, $\gamma \in \mathds{T}$, $\gamma \neq 1$. Then $Ax \neq 0$ and
\begin{equation}\label{eq24}
  TAx = wUAX = wUAU^{-1}Ux = \gamma wAUx = \gamma ATx = \gamma \lambda Ax.
\end{equation}
We consider two possibilities.

(I) $\gamma$ is a root of unity. Let $q \in \mathds{N}$, $q > 1$, and $\gamma^q = 1$.
Then it follows from~(\ref{eq24}) that $T^q Ax = \lambda x$. But then in virtue of
minimality of the system $(K, \varphi^q)$ and part (1) of the proof we have that $Ax =
sX$, $s \in \mathds{C} \setminus \{0\}$, a contradiction.

(II) $\gamma$ is not a root of unity. It follows from~(\ref{eq24}) that $\lambda
\mathds{T} \subseteq \sigma_{a.p.}(T)$. Then there is an open subinterval $I$ of
$\lambda \mathds{T}$ such that $\lambda \in I \subseteq \sigma_p(T)$. Let $\alpha$ be a
root of unity such that $\alpha \neq 1$ and $\alpha \lambda \in I$. It remains to
repeat the argument from part (I) above.
\end{proof}

We proceed with applying Theorems~\ref{t19} and~\ref{t25} to special situations when
$\Omega$ is the unit disc $\mathds{U}$, a polydisc $\mathds{U}^n$, or a unit ball
$\mathds{B}^n$.

\begin{corollary} \label{c6} Let $X$ be a Banach space of functions analytic in
$\mathds{U}^n$. Let $\boldsymbol{\alpha} = (\alpha_1, \ldots \alpha_n) \in
\mathds{T}^n$ be
such that
\begin{equation}\label{eq25}
  \alpha_1^{m_1} \ldots \alpha_n^{m_n} =1, m_1, \ldots , m_n \in \mathds{Z}
  \Leftrightarrow m_j =0, j = 1, \ldots , n
\end{equation}
 and $w$, $w \not \equiv 0$, be a function analytic in $\mathds{U}^n$. Assume that the
 operator $T$,
$$(Tx)(\boldsymbol{z}) =w(\boldsymbol{z})x(\boldsymbol{\alpha} \boldsymbol{z}), x \in
X, \boldsymbol{z} \in \mathds{U}^n, $$
is defined and bounded on $X$. Assume additionally that there are $j \in [1 : n]$ and
$k \in \mathds{N}$ such that $z_j^k X \subset X$. Then
\begin{equation*}
  \sigma_1(T) = \sigma_{a.p.}(T).
\end{equation*}
\end{corollary}

\begin{proof} Fix an $r \in (0,1)$ and let $K = r \mathds{T}^n$ and
$\varphi(\boldsymbol{z}) = \boldsymbol{\alpha} \boldsymbol{z}, \boldsymbol{z} \in
\mathds{U}^n$. It follows from~(\ref{eq25}) that for any $m \in \mathds{N}$ the system
$(K, \varphi^m)$ is minimal. The measure $\mu$ on $K$ is defined in an obvious way: if
$E$ is a Borel subset of $K$ then $\mu(E) = m_n(\frac{1}{r}E)$ where $m_n$ is the Haar
measure on $\mathds{T}^n$. It is immediate to see that all the conditions of
Theorem~\ref{t25} are satisfied.
\end{proof}

Assume conditions of Corollary~\ref{c6} and let $\lambda \in \sigma_p(T)$. It follows
from Theorem~\ref{t25} that $\dim \ker {(\lambda I - T)} = 1$. Some additional
information about the point spectrum of weighted rotation operators in spaces of
analytic functions is contained in the next corollary.

\begin{corollary} \label{c8} Let $\Omega$ be an open connected rotation invariant
subset of $\mathds{C}$. Assume also that $0 \in \Omega$ and let $R$ be the largest
positive number such that $R\mathds{U} \subseteq \Omega$. Let $\alpha \in \mathds{T}$
and let $w \in \mathcal{H}(\Omega)$, $w \not \equiv 0$. Let $T$ be the weighted
rotation operator in $\mathcal{H}(\Omega)$,
\begin{equation*}
  (Tx)(\omega) = w(\omega)x(\alpha \omega), x \in \mathcal{H}(\Omega), \omega \in
  \Omega .
\end{equation*}
Let $\lambda \in \sigma_p(T)$. Then

\begin{enumerate}[(a)]
\item $w \neq 0$ in $R\mathds{U}$,
\item $\lambda \in \{\alpha^k w(0): k = 0,1, \ldots \}$.
 \end{enumerate}
\end{corollary}

\begin{proof} (a) By~(\ref{eq28})
\begin{equation}\label{eq26}
   |\lambda| = \exp \frac{1}{2\pi} \int \limits_0^{2\pi} \ln |w(re^{i\theta})|d\theta
\end{equation}
for any $r \in (0,R)$. If $w(0) \neq 0$ then it follows from~(\ref{eq26}) and Jensen's
formula (see e.g.~\cite{Ah}) that $w$ cannot have zeros in $\mathds{U}$. Let $w(0) =
0$. Then $w(z) = z^k w_1(z)$ where $k \in \mathds{N}$  and $w_1(0) \neq 0$, and
therefore
\begin{equation} \label{eq27} |\lambda| = \exp \frac{1}{2\pi} \int \limits_0^{2\pi} \ln
|w_1(re^{i\theta})|d\theta +k \ln r.
\end{equation}
Combining~(\ref{eq27}), Jensen's formula for $w_1$, and~(\ref{eq28})  we come to a
contradiction.

\noindent (b) Let $x \in \mathcal{H}(\Omega) \setminus \{0\}$ and $Tx = \lambda x$. If
$x(0) \neq 0$ then it follows from (a) that $\lambda = w(0)$. If, on the other hand,
$w(0) = 0$ then for some $k \in \mathds{N}$, $x(z) = z^k x_1(z)$ where $x_1(z) \neq 0$.
Hence,
\begin{equation*}
  (Tx)(z) = \alpha^k w(z)z^k x_1(\alpha z) = \lambda z^k x_1(z).
\end{equation*}
Therefore $Tx_1 = \bar{\alpha}^k \lambda x_1$ and $\lambda = \alpha^k w(0)$.
\end{proof}

We can now prove a result similar to Corollary~\ref{c8} for functions analytic in a
domain in $\mathds{C}^n$.

\begin{corollary} \label{c9} Let $\Omega$ be an open connected subset of $\mathds{C}^n$
such that $\mathbf{0} \in \Omega$ and for any $\boldsymbol{\alpha} = (\alpha_1, \ldots
, \alpha_n) \in \mathds{T}^n$ and for any $\mathbf{z} = (z_1, \ldots , z_n) \in \Omega$
we have \footnote{We do not call $\Omega$ rotation invariant because usually this term
is reserved for domains invariant under all linear unitary transformations of
$\mathds{C}^n$.}
\begin{equation*}
  \boldsymbol{\alpha z} = (\alpha_1 z_1, \ldots , \alpha_n z_n) \in \Omega.
\end{equation*}
Let $w \in \mathcal{H}(\Omega)$, $w \not \equiv 0$ and let $T$ be the weighted rotation
operator in $\mathcal{H}(\Omega)$,
\begin{equation*}
  (Tx)(\boldsymbol{\omega}) = w(\boldsymbol{\omega})x(\boldsymbol{\alpha \omega}), x
  \in \mathcal{H}(\Omega), \omega \in \Omega .
\end{equation*}
Let $\lambda \in \sigma_p(T)$. Then

\begin{enumerate}[(a)]
\item $w(\mathbf{0}) \neq 0$,
\item $\lambda \in \{\alpha_1^{k_1} \ldots \alpha_n^{k_n} w(\mathbf{0}): k_j = 0,1,
    \ldots , j= 1, \ldots , n. \}$.
 \end{enumerate}
\end{corollary}

\begin{proof} We will prove the corollary by induction. For $n=1$ the statement follows
from Corollary~\ref{c8}. Assume that our claim is true for $n-1$. Let $x \in
\mathcal{H}(\Omega) \setminus \{0\}$ and $Tx = \nu x$. Consider $\tilde{x}(z_2, \ldots
, z_n) = x(0, z_2, \ldots , z_n$. there are two possibilities.

\noindent (I). $\tilde{x} \not \equiv 0$. Then it is immediate to see that our claim
follows from the induction assumption.

\noindent (II).  $\tilde{x} \equiv 0$. In this case $x = z_1^{k_1}y, k_1 \in
\mathds{N}$, and  $y(0, z_2, \ldots , z_n) \not \equiv 0$. It remains to notice that
$Ty = \bar{\alpha}_1^{k_1} \nu y$ and apply the induction assumption.
\end{proof}

It is possible now to improve the result of Corollary~\ref{c6}.

\begin{theorem} \label{t26} Let $\Omega$ be an open connected subset of $\mathds{C}^n$
such that $\mathbf{0} \in \Omega$ and for any $\boldsymbol{\alpha} = (\alpha_1, \ldots, \alpha_n) \in \mathds{T}^n$, $\mathbf{z} = (z_1, \ldots , z_n) \in \Omega$
we have
\begin{equation*}
  \boldsymbol{\alpha z} = (\alpha_1 z_1, \ldots , \alpha_n z_n) \in \Omega.
\end{equation*}
Let us fix $\boldsymbol{\alpha} = (\alpha_1, \ldots \alpha_n) \in \mathds{T}^n$. Assume
that there is $j \in [1:n]$ such that $\alpha_j$ is not a root of unity. Let
  $w \in \mathcal{H}(\Omega)$, $w \not \equiv 0$. Let $X$ be a Banach space of
  functions analytic in $\Omega$. Assume that the operator $T$,
$$(Tx)(\boldsymbol{z}) =w(\boldsymbol{z})x(\boldsymbol{\alpha} \boldsymbol{z}), x \in
X, \boldsymbol{z} \in \Omega, $$
is defined and bounded on $X$. Assume additionally that there is $k \in \mathds{N}$
such that $z_j^k X \subset X$. Then
\begin{equation*}
  \sigma_1(T) = \sigma_{a.p.}(T).
\end{equation*}
\end{theorem}

\begin{proof} Let $\lambda \in \sigma_{a.p.}(T) \setminus \sigma_1(T)$. Then $\lambda
\in \sigma_p(T)$. Let $x \in X \setminus \{0\}$ be such that $Tx = \lambda x$. Then
for any $m \in \mathds{N}$ we have $T(z_j^{km} x) = \bar{\alpha}_j^{km} z_j^{km} x$.
Therefore $\lambda \mathds{T} \subseteq \sigma_{a.p.}(T)$ and the set $\sigma_{a.p.}(T)
\setminus \sigma_1(T) \subseteq \sigma_p(T)$ is uncountable in contradiction with
Corollary~\ref{c9} (b).
\end{proof}

\begin{remark} \label{r12} Assume conditions of Theorem~\ref{t26}. It follows that in
order to obtain a complete description of essential spectra of $T$ it is sufficient to
\begin{enumerate}
  \item Describe $\sigma(T)$.
  \item Describe $\sigma_{a.p.}(T)$.
  \item Find $\mathrm{codim}(\lambda I -T)X$ for any $\lambda \in \sigma_r(T)$.
\end{enumerate}

\end{remark}

\section{\centerline{ Uniform algebras. Polydisc and ball algebras}}

\begin{definition} \label{d5} Let $K$ be a compact Hausdorff space and $\varphi$ be a
homeomorphism of $K$ onto itself. A nonempty subset $E$ of $K$ is called
$\varphi$-wandering if the sets $\varphi^i(E), i \in \mathds{Z}$ are pairwise disjoint.
\end{definition}

\begin{theorem} \label{t15} Let $A$ be a unital uniform algebra. Let $U$ be an
automorphism of $A$ and $\varphi$ be the corresponding homeomorphism of
$\mathfrak{M}_A$ (and $\partial A$) onto itself. Let $w \in A$ and $T = wU$. Assume
that
\begin{enumerate}[(a)]
  \item The set of all $\varphi$-periodic points is of first category in $\partial
      A$.
  \item There are no open $\varphi$-wandering subsets of $\partial A$.
  \end{enumerate}
  Then
  \begin{enumerate}
    \item $\sigma_1(T) = \sigma_{a.p.}(T)$ and $\sigma(T) = \sigma_4(T)$.
    \item The sets $\sigma_i(T), i = 1, \ldots , 5$ are rotation invariant subsets of
        $\mathds{C}$.
    \item If $\mathfrak{M}_A$ is not the union of two nonempty clopen
        $\varphi$-invariant subsets then $\sigma(T)$ is a disk centered at $0$.
    \item If, moreover, $\partial A$ is not the union of two nonempty clopen
        $\varphi$-invariant subsets then there are three possibilities.
        \begin{enumerate}[(I).]
          \item If $w$ is invertible in $A$ then $\sigma(T) = \sigma_1(T)$ is an
              annulus or a circle centered at $0$.
         \item If $w$ is not invertible in $C(\partial A)$ then $\sigma(T) =
             \sigma_1(T)$ is either a disk centered at $0$ or the singleton
             $\{0\}$.
          \item If $w$ is invertible in $C(\partial A)$ but not invertible in $A$
              then $\sigma_1(T) = \sigma_{a.p.}(T)$ is an annulus or a circle
              centered at $0$ and $\sigma_r(T)$ is the open disc $r\mathds{U}$
              where
              $$ r = \min \exp \int \ln |w| d\mu, $$
              where the minimum is taken over the set of all $\varphi$-invariant
              regular Borel probability measures on $\partial A$.
        \end{enumerate}
  \end{enumerate}
\end{theorem}

\begin{proof} The only part of the statement of the theorem that requires a proof is
that $\sigma_1(T) = \sigma_{a.p.}(T)$. The rest follows immediately from
Theorems~\ref{t1} and~\ref{t2}.

\noindent Part (A). We will prove that $\sigma_2(T) = \sigma_{a.p.}(T)$.  Let $\lambda
\in \sigma_{a.p.}(T)$.

 \noindent Assume first that $\lambda = 0$. Then the set $Z(w) = \{k \in \partial A :
 w(k) = 0\} \neq \emptyset$. Conditions (a) and (b) combined guarantee that $\partial
 A$ has no isolated points and therefore we can find open nonempty pairwise disjoint
 subsets $V_n, \; n \in \mathds{N}$, of $\partial A$ such that $|w| < 1/n$ on $V_n$.
 Let $f_n \in A$, $\|f_n\| = 1$, and $|f_n(k)| < 1/n, k \in \partial A \setminus V_n$.
 Then $Tf_n \rightarrow 0$. Notice that the sequence $f_n$ is singular. Indeed,
 otherwise there is a subsequence $f_{n_k}$ of $f_n$ such that $f_{n_k} \rightarrow f
 \in A$. Then $\|f\| = 1$ and $f_{n_k}f_{n_{k+1}} \rightarrow f^2$. On the other hand,
 it follows from the construction of the sequence $f_n$ that $f_{n_k}f_{n_{k+1}}
 \rightarrow 0$, a contradiction.

\noindent Assume now that $\lambda \neq 0$. Without loss of generality we can assume
that $\lambda = 1$. By Lemma~\ref{l4} there is $k \in \partial A$ such that
\begin{equation}\label{eq29}
  |w_m(k)| \geq 1, \; |w_m(\varphi^{-m}(k)| \leq 1, \; m \in \mathds{N}.
\end{equation}
It follows from (a) and~(\ref{eq29}) that we can find nonempty open subsets $V_n, n \in
\mathds{N}$ of $\partial A$ with properties.

\noindent $(\alpha)$ The sets $\varphi^m(V_n), n \in \mathds{N}, |m| \leq n$ are
pairwise disjoint.

\noindent $(\beta)$ For any $n \in \mathds{N}$ and for any $s \in V_n$,
\begin{equation}\label{eq30}
  |w_m(s)| \geq \frac{1}{2}, \; |w_m(\varphi^{-m}(k)| \leq 2, \;
  |m| \leq n.
\end{equation}
  Let $f_n \in A$ be such that $\|f_n\| = 1$ and
  \begin{equation}\label{eq31}
    |f_n(t)| < \frac{1}{n \sum \limits_{i=0}^{2n+1} \|T^i\|}, \; t \in \partial A
    \setminus \varphi^n(V_n).
  \end{equation}
  Let $h_n = S_n(T, 1/\sqrt{n}) h_n$ (see Definition~\ref{d7}). It follows
  from~(\ref{eq30}), ~(\ref{eq31}), and Lemma~\ref{l3} that $\|h_n\| \geq 1/2$, $Th_n -
  h_n \rightarrow 0$, and $h_n h_{n+1} \rightarrow 0$. Thus, $1 \in \sigma_2(T)$.
  \bigskip

  \noindent Part (B). Now we will prove that $\sigma_2(T^\prime) = \sigma_{a.p.}(T)$.
  Let again $\lambda \in \sigma_{a.p.}(T)$. Notice that if $s_1, \ldots , s_r$ are
  distinct points in $\partial A$ then it follows easily from the definition of Shilov
  boundary that
  $$ \| \sum \limits_{j=1}^r a_i \delta_{s_i}\|_{A^\prime} = \| \sum \limits_{j=1}^r
  a_i \delta_{s_i}\|_{C(\partial A)^\prime} = \sum \limits_{i=1}^r |a_i|. $$

   The case $\lambda = 0$ is rather obvious. Indeed, the conditions of the theorem
   guarantee that there are pairwise distinct points $k_n \in \partial A$ such that
   $|w(k_n)| < 1/n, n \in \mathds{N}$. Let $\delta_s(f) = f(s), f \in A, s \in \partial
   A$. Then $\|\delta_{k_n}\|_{A^\prime} = 1$, $T^\prime \delta_{k_n} = w(k_n)
   \delta_{\varphi(k_n)} \rightarrow 0$, and $\|\delta_{k_n} -
   \delta_{k_m}\|_{A^\prime} = 2, n < m \in \mathds{N}$. Thus, the sequence
   $\delta_{k_n}$ is singular.

   If $\lambda \neq 0$ we again can assume without loss of generality that $\lambda =
   1$. Applying Lemma~\ref{l4} to the operator $wU^{-1}$  we see that there is $p \in
   \partial A$ such that
  \begin{equation}\label{eq32}
     |w_m(p)| \leq 1, \; |w_m(\varphi^{-m}(p)| \geq 1, \; m \in \mathds{N}.
  \end{equation}
 It follows from~(\ref{eq32}) and (a) that there are $p_n \in \partial A$ such that

 $(\gamma)$ All the points $\varphi^m(p_n), n \in \mathds{N}, |m| \leq n$ are pairwise
 distinct.

 $(\delta)$
  \begin{equation}\label{eq33}
     |w_m(k)| \leq 2, \; |w_m(\varphi^{-m}(p_n)| \geq 1/2, \; n \in \mathds{N}, |m|
     \leq n.
  \end{equation}
  Let $\mu_n= S_n (T^\prime, 1/\sqrt{n}) \delta_{\varphi^{-n}(p_n)}$ and let $\nu_n
  =\mu_n/\ \|\mu_n\|_{A^\prime}$. Then it follows from~(\ref{eq33}) and Lemma~\ref{l3}
  that $T\nu_n - \nu_n \rightarrow 0$. It remains to notice that
  for any $n, n^\prime \in \mathds{N}$ we have $\|\nu_n - \nu_{n^\prime}\| =2$ and
  therefore the sequence $\nu_n$ is singular.
\end{proof}

\begin{remark} \label{r5} Condition (b) in Theorem~\ref{t15} is satisfied, in
particular, if there is a $\varphi$-invariant regular Borel probability measure $\mu$
on $C(\partial A)$ such that for any nonempty open subset $O$ of $\partial A$ we have
$\mu(O) > 0$.
\end{remark}

\begin{example} \label{e4} Let $n \in \mathds{N}$ and $A(\mathds{U}^n)$ be the Banach
algebra of all functions analytic in the polydisc $\mathds{U}^n$ and continuous on
$\mathds{D}^n$. It is well known (see e.g.~\cite{AW} or~\cite{Kan}) that the space
of maximal ideals of the algebra $A(\mathds{U}^n)$ can be identified with
$\mathds{D}^n$ and its Shilov boundary with $\mathds{T}^n$.

Recall that a M\"{o}bius transformation of $\mathds{U}$ onto itself is called
\textit{elliptic} if it has a fixed point in $\mathds{U}$. If  $\varphi$ is an elliptic
M\"{o}bius transformation then there is another M\"{o}bius transformation $\psi$ such
that $\psi^{-1} \circ \varphi \circ \psi$ is a rotation of $\mathds{U}$.

Let $\Pi$ be a permutation of the set $[1 : n]$, $\varphi_i, i \in [1:n]$ be elliptic
M\"{o}bius transformations on $\mathds{U}$ such that for at least one $j \in [1:n]$ the
map $\varphi_j$ is not periodic. Let $w \in A(\mathds{U}^n$, let $\Phi : \mathds{U}^n
\rightarrow \mathds{U}^n$ be defined as
$$ \Phi(z_1, \ldots , z_n) = (\varphi_{\Phi(1)}(z_{\Phi(1)}), \ldots ,
\varphi_{\Phi(n)}(z_{\Phi(n)})). $$
Finally, let $Tf = w(f \circ \varphi), f \in A(\mathds{U}^n)$.

 Then by Theorems~\ref{t1},~\ref{t2}, and~\ref{t15} the sets $\sigma(T)$ and
 $\sigma_1(T) = \sigma_{a.p.}(T)$ are rotation invariant connected subsets of
 $\mathds{C}$.

 Moreover, if $w$ is invertible in $C(\mathds{T}^n)$ but not invertible in $A(U^n)$
 then

 \noindent (a) If $n =1$ then $\sigma_3(T) = \sigma_1(T)$ and $\sigma_4(T) =
 \sigma(T)$.

\noindent (b) If $n > 1$ then $\sigma(T) = \sigma_3(T)$.

Indeed, let $\lambda \in \sigma_r(T) = \sigma(T) \setminus \sigma_1(T)$. Then
$ind(\lambda I - T) = ind(T)$.

\noindent If $n =1$ then $w$ has a finite number of zeros in $\mathds{U}$ and therefore
$codim(TA(\mathds{U})) < \infty$.

\noindent If $n > 1$ and $w(z) = 0$, $z \in \mathds{U}^n$ then $codim(TA(\mathds{U})) =
\infty$ because $w$ cannot have isolated zeros in $\mathds{U}^n$.

\noindent If $n >1$ and $w$ has no zeros in $\mathds{U}^n$ then it must have a zero,
$z$ in $\partial \mathds{U}^n \setminus \mathds{T}^n$. Let $z_n \in \mathds{U}^n, z_n
\rightarrow z$ and let $T_nf = (w - w(z_n)(f\circ \varphi), f \in A(\mathds{U}^n)$.
Then $ind(T_n) = \infty$ and, because index is stable under small norm perturbations,
$ind(T) = \infty$.

Consider the following special case. Let $\varphi$ be a non-periodic rotation of
$\mathds{U}^n$, i.e. for any $(z_1, \ldots , z_n) \in \mathds{U}^n$
$$ \varphi(z_1, \ldots , z_n) = (\alpha_1 z_1, \ldots , \alpha_n z_n), $$
where $\alpha_i \in \mathds{T}, i= 1, \ldots , n$, and for at least one $i$, $\alpha_i$
is not a root of unity. Let $w \in A(\mathds{U}^n)$, $U$ be the composition operator,
$Uf = f\circ \varphi, f \in A(\mathds{U}^n)$, and $T = wU$.

Let $H$ be the closed subgroup of $\mathds{T}^n$ generated by $\boldsymbol{\alpha}$ and $m_H$ be the Haar measure on $H$.
For any $\mathbf{t} \in \mathds{T}^n$ let $U_{\mathbf{t}}f(z) = f(\mathbf{t}z), f \in
A(\mathds{U}^n), z \in \mathds{U}^n $. Then
\begin{enumerate}
\item $\rho(T) > 0$.
\item By Theorem~\ref{t13}
$$ \rho(T) = \max \limits_{\mathbf{t} \in \mathds{T}^n} \exp \int \ln
|U_{\mathbf{t}}w| dm_H.$$
\item In particular, if the condition~(\ref{eq25}) is satisfied then
  $$ \rho(T) = \exp \int \ln |w| dm_n,$$
  where $m_n$ is the Haar measure on $\mathds{T}^n$.
 \end{enumerate}

\end{example}

\begin{example} \label{e2} Let $n >1$ and $A = A(\mathds{B}^n)$ be the Banach algebra
of all functions analytic in the unit ball $\mathds{B}^n$ of $\mathds{C}^n$ and
continuous in $cl(\mathds{B}^n)$. Then (\cite{AW}, ~\cite{Kan}) $\mathfrak{M}_A =
\mathds{B}^n$ and $\partial A = \partial \mathds{B}^n$ - the topological boundary of $\mathds{B}^n$ in $\mathds{C}^n$. Let $\tau$ be a linear unitary
transformation of $\mathds{C}^n$ such that for any $n \in \mathds{N}$, $\tau^n \neq I$.
Let $w \in A(\mathds{B}^n)$ and
$$ Tf = w(f \circ \tau), f \in A(\mathds{B}^n).$$
Then
\begin{enumerate}
 \item If $w$ is an invertible element of $A(\mathds{B}^n)$ then $\sigma(T) =
     \sigma_1(T)$ and $\sigma(T)$ is either an annulus or a circle centered at $0$.
  \item If $w$ is not invertible in $C(\partial \mathds{B}^n)$ then $\sigma_1(T) =
      \sigma(T)$ is a disc centered at $0$.
  \item If $w$ is invertible in $C(\partial{B}^n)$ but not invertible in
      $A(\mathds{B}^n)$ then $\sigma(T)$ is a disc centered at $0$, $\sigma_1(T) =
      \sigma_{a.p.}(T)$ is an annulus or a circle centered at $0$, $\sigma_r(T)$ is
      an open disc centered at $0$ and $\sigma(T) = \sigma_3(T)$.
\end{enumerate}
\end{example}

\begin{example} \label{e3} Let $K$ be a rotation invariant compact subset of
$\mathds{C}$. Let $K^0 = Int_{\mathds{C}}K$. We consider the Banach algebra $A(K)$ of
all functions continuous on $K$ and analytic in $K^0$. It is known that the space of
maximal ideals of $A(K)$ is $K$ (see~\cite[Theorem 2.6.6, p. 88]{Kan}) and that the
Shilov boundary of $A(K)$ is $\partial K$ - the topological boundary of $K$ in
$\mathds{C}$ (It follows from~\cite[Example 3.3.5 (2), p. 161]{Kan}).  Let $\alpha \in
\mathds{T}$. Assume that $\alpha$ is not a root of unity. Let $w \in A(K)$ and let
$$Tf(z) = w(z) f(\alpha z), f \in A(K), z \in K. $$
Then it follows from Theorem~\ref{t15} that

\noindent (1) $\sigma_1(T) = \sigma_{a.p.}(T) = \sigma(T, C(\partial K))$.

\noindent (2) The sets $\sigma_i(T), i = 1, \ldots , 5$ are rotation invariant.

\noindent (3) $\sigma(T) = \sigma_4(T)$.

\end{example}

\begin{remark} \label{r6} Assume conditions of Example~\ref{e3}. It follows from the
fact that $K$ is rotation invariant and the celebrated Vitushkin's theorem (see
e.g.~\cite[Theorem 8.2]{Ga}) that $A(K) = R(K)$, where $R(K)$ is the closure in $C(K)$
of the algebra of all rational functions with poles in $\mathds{C} \setminus K$. We are
grateful to Professor O. Eroshkin for the corresponding information.
\end{remark}

\begin{problem} \label{pr6} Assume conditions of Example~\ref{e3}. Describe
$\sigma_3(T)$. Equivalently, assume that $\lambda \in \sigma_r(T)$; find necessary
and/or sufficient conditions for $codim (\lambda I - T)A(K) = \infty$.

We do not know the answer to Problem~\ref{pr6} even in the following situation. Let
$K$ be the annulus $K = \{z \in \mathds{C}: a \leq |z| \leq b\}$, where $0 < a < b <
\infty$. Let $\alpha \in \mathds{T}$ be not a root of unity and let
$$ (Tf)(z) = zf(\alpha z), f \in A(K), z \in K. $$
Then $\sigma(T) = K$ and $\sigma_{a.p.}(T) = \sigma_1(T) = \partial K$. Let $\lambda
\in Int_{\mathds{C}}K$. What is $codim(\lambda I - T)A(K)$?

\end{problem}

\begin{example} \label{e5} Here we consider the Banach algebra $H^\infty(\mathds{U})$
of all functions analytic and bounded in the unit disc $\mathds{U}$.

\begin{corollary} \label{c10} Let $T$ be a weighted rotation operator
$$(Tf)(z) = w(z)f(\alpha z), f \in H^\infty(\mathds{U}), z \in \mathds{U},$$
where $w \in H^\infty(\mathds{U})$ and $\alpha \in \mathds{T}$ is not a root of unity.
Then
\begin{enumerate}
  \item If $w \in (H^\infty(\mathds{U}))^{-1}$ then $\sigma(T) = \sigma_1(T)$ is
      either an annulus or a circle centered at $0$.
  \item If $w^\star$ is not invertible in $L^\infty(\mathds{T})$, where for almost
      all $t \in \mathds{T}$ $w^\star(T)$ is the radial limit of $w$ (see~\cite{Ho}),
      then $\sigma(T) = \sigma_1(T)$ is a disc centered at $0$.
  \item If $w \not \in (H^\infty(\mathds{U}))^{-1}$ but $w^\star$ is invertible in
      $L^\infty(\mathds{T})$ then $\sigma_1(T) = \sigma_{a.p.}(T)$ is an annulus or a
      circle centered at $0$, while $\sigma(T) = \sigma_4(T)$ is a disk.

  \item Assume conditions of case (3) above. Then $\sigma_3(T) = \sigma_1(T)$ if and
      only if $w =B w_e$ where $B$ is a finite Blaschke product and $w_e$ is the
      outer function in the canonical factorization of $w$ (\cite{Ho}). Otherwise
      $\sigma_3(T) = \sigma(T)$.
  \item If $w^\star$ is Riemann integrable then
  \begin{equation*}
    \rho(T) = |w_e(0)|
  \end{equation*}
\end{enumerate}
\end{corollary}

\begin{proof} (1) - (3) follow from Theorem~\ref{t15} and the well known fact that the
Shilov boundary of $H^\infty(\mathds{U})$ can be identified with the space of maximal
ideals of the algebra $L^\infty(\mathds{T})$ (see e.g.~\cite[p. 174]{Ho}).

\noindent (4) In virtue of (3) the index of the operator $\lambda I - T$ is constant in
the open disc $\sigma(T) \setminus \sigma_1(T)$ and coincides with $-(\mathrm{codim}\;
TX)$. Consider the canonical factorization (\cite[p. 67]{Ho}) $w = BSw_e$ where $B$ is a
Blaschke product, $S$ - a singular inner function, and $w_e$ - an outer function. If
either $B$ is an infinite Blaschke product or the factor $S$ is nontrivial then it is
immediate to see that $\mathrm{codim} \; TX = \infty$. Otherwise $\mathrm{codim}\; TX$
is equal to the number of zeros of $B$ taking into consideration their multiplicities.

\noindent (5) follows from Corollary~\ref{c4}.
\end{proof}

\begin{remark} \label{r13} Statement (5) of Corollary~\ref{c10} is not valid for an
arbitrary $w \in H^\infty(\mathds{U})$. Indeed, et $\hat{w} \in L^\infty(\mathds{T})$
be the function constructed in Example~\ref{e1}. Then (see e.g.~\cite[p. 53]{Ho}) there is an
invertible $w \in H^\infty(\mathds{U})$ such that the boundary values of $w$ on
$\mathds{T}$ coincide a.e. with $\hat{w}$, and therefore $\sigma(T) = \sigma_1(T)$ is
an \textbf{annulus}.
\end{remark}

\begin{problem} \label{pr10} Is an analogue of Corollary~\ref{c10} true for weighted
rotation operators on $H^\infty(\mathds{U}^n)$ or $H^\infty(\mathds{B}^n)$ for $n > 1$?
\end{problem}
\end{example}

\bigskip

\section{Spectra of weighted rotation-like operators on Banach spaces of analytic
functions}

\bigskip

\subsection{\centerline{Hardy - Banach spaces}}

In this subsection we will extensively use the fact that the Hardy space,
$H^p(\mathds{U})$ where $1 \leq p < \infty$ is a "rich" closed subspace of
$L^p(\mathds{T)}$. Namely, if $x$ is a nonnegative function from $L^p(\mathds{T})$ such
that $\int_0^{2\pi} \ln |x(e^{i\theta})|d\theta > -\infty$ then there is a $y \in
H^p(\mathds{U})$ such that $|y| = x$ (see~\cite[p. 53]{Ho}). It leads to the folowing
definition.

\begin{definition} \label{d6} (~\cite[Definition 19]{Ki2}) Let $X$ be a Dedekind
complete Banach lattice and $Y$ be a closed subspace of $X$. We say that $Y$ is
\textit{almost localized} in $X$ if for any band $Z$ in $X$ and for any $\varepsilon >
0$ there is a $y \in Y$ such that $\|y\|=1$ and $\|(I - P_Z)y\| < \varepsilon$ where
$P_Z$ is the band projection on $Z$.
\end{definition}

\begin{theorem} \label{t16} Let $K$ be a Hausdorff compact space, $\varphi$ be a
homeomorphism of $K$ onto itself, and $\mu$ be a $\varphi$-invariant regular Borel
probability measure on $K$ such that for any nonempty open subset $O$ of $K$ we have
$\mu(O) > 0$. Let $X$ be a Banach ideal in $L^0(K,\mu)$ such that $Z(X) = L^\infty(K,
\mu)$, the composition operator $Ux = x \circ \varphi$ is defined and bounded on $X$,
and moreover, $\sigma(U) \subseteq \mathds{T}$. Assume also that $\mu(\Pi) = 0$ where
$\Pi$ is the set of all $\varphi$-periodic points in $K$.

Let $Y$ be almost localized closed subspace of $X$ such that $UY = Y$. Let $w \in
L^\infty(K,\mu)$ be such that $wY \subseteq Y$, and let $T = wU$. Then

(1) $ \sigma_2(T) = \sigma_{a.p.}(T)$ is a rotation invariant subset of $\mathds{C}$.

\noindent Assume additionally that for any nonzero $y \in Y$ we have
\begin{equation}\label{eq36}
  \Big{|} \int \ln |y| d\mu \Big{|} < \infty.
\end{equation}
Then

(2) $ \sigma_1(T) = \sigma_{a.p.}(T)$ and $\sigma(T) = \sigma_4(T)$.

\end{theorem}

\begin{proof} (1) The proof of Theorem 20 in~\cite{Ki2} shows that $\sigma_{a.p.}(T) =
\sigma_{a.p.}(T,X)$ and moreover that if $\lambda \in \sigma_{a.p.}(T,X)$ the we can
construct a sequence $y_n \in Y$ such that $\|y_n\| = 1$, $Ty_n - \lambda y_n
\rightarrow 0$, and $|y_n| \wedge |y_{n+1}| \rightarrow 0$ in $X$, which obviously
implies that the sequence $y_n$ is singular and therefore $\sigma_2(T) =
\sigma_{a.p.}(T)$.

The condition $\mu(\Pi) = 0$ guarantees that the set $\sigma_{a.p.}(T)$, and therefore
$\sigma(T)$, is rotation invariant.

The equality $\sigma(T,X) = \sigma_{a.p.}(T,X)$ follows from Theorem 4.5 in~\cite{Ki3}.

 (2). Assume that $\lambda \in \sigma_{a.p.}(T) \setminus \sigma_2(T^\prime)$.
We have to consider two possibilities.

(a) $\lambda \in \partial \sigma_{a.p.}(T)$ - the topological boundary of
$\sigma_{a.p.}(T)$ in $\mathds{C}$. Then, applying Lemma~\ref{l1} and
Theorem~\ref{t22}, we come to a contradiction.

(b) $\lambda \in Int_{\mathds{C}} \sigma_{a.p.}(T)$. Recalling that $\sigma_{a.p.}(T)$
is rotation invariant and that the set $\sigma_{a.p.}(T) \setminus \sigma_2(T^\prime)$
is open in $\mathds{C}$ we see that the set $\{ |\lambda| : \lambda \in
\sigma_{a.p.}(T) \setminus \sigma_2(T^\prime)\}$ contains an interval $[a,b]$, $0 \leq
a < b$. On the other hand if $\lambda \in \sigma_{a.p.}(T) \setminus
\sigma_2(T^\prime)$ then $\lambda$ must be an eigenvalue of $T$. Indeed, otherwise we
would have $\lambda \in \sigma_r(T)$, a contradiction. Let $y \in Y \setminus \{0\}$ be
such that $Ty = \lambda y$. Then
\begin{equation}\label{eq37}
   \int\ln |w| d\mu + \int \ln|Uy| d\mu = \ln |\lambda| + \int \ln |y| d\mu .
\end{equation}
Condition~(\ref{eq36}) guarantees that all the integrals in~(\ref{eq37}) converge, and
because $\mu$ is a $\varphi$-invariant measure we have
$$ |\lambda| = \int \ln |w| d\mu , $$
a contradiction.
\end{proof}

\begin{problem} \label{pr4} Does statement (2) of Theorem~\ref{t16} remain true without
assuming condition~(\ref{eq36})?
\end{problem}

\begin{example} \label{e6} Let $m$ be the normalized Lebesgue measure on $\mathds{T}$.
Let $X$ be a Banach ideal space such that $L^\infty(\mathds{T}, m) \subseteq X
\subseteq L^1(\mathds{T}, m)$. Assume that $\alpha \in \mathds{T}$ is not a root of
unity. Let $\varphi(z) = \alpha z, z \in \mathds{T}$. Assume that the operator $U$, $Ux
= x \circ \varphi$, is bounded on $X$ and that $\sigma(U) \subseteq \mathds{T}$. Notice
that this condition is automatically satisfied if $X$ is an interpolation space between
$L^\infty(\mathds{T}, m)$ and $L^1(\mathds{T}, m)$.

Let us identify the Hardy space $H^1(\mathds{U})$ with a closed subspace of
$L^1(\mathds{T, m)}$ and let $Y = H^1(\mathds{U}) \cap X$. Then $Y$ is a closed
subspace of $X$ and $Y$ consists of all functions from $H^1(\mathds{U})$ with boundary
values in $X$. Notice that operator $U$ acts on $Y$ and $\sigma(U, Y) \subseteq
\mathds{T}$. Let $w \in H^\infty(\mathds{U})$ and let $T = wU : Y \rightarrow Y$.

We claim that all the conditions of Theorem~\ref{t16} are satisfied. Let $Z$ be a band
in $X$. Than there is a measurable subset $E$ of $\mathds{T}$ such that $m(E) > 0$ and
$Z = \chi_E X$. Let $\|\chi_E\|_X = C$. Fix $\varepsilon > 0$. Let $g =
\frac{1}{C}\chi_E + \varepsilon \chi_{\mathds{T} \setminus E}$. There is
(see~\cite[p.53]{Ho}) $y \in H^\infty(\mathds{U}) \subseteq Y$ such that $|y|$ coincide
a.e. on $\mathds{T}$ with $g$. It follows that $Y$ is almost localized in $X$.

Condition~(\ref{eq36}) is satisfied for any nonzero function from
$H^1(\mathds{U})$ (see e.g.~\cite[p.51]{Ho}).

As a result we obtain the following corollary which in virtue of Corollary~\ref{c10}
provides a complete description of essential spectra of $T$ on $Y$.

\begin{corollary} \label{c11} Let the operator $T$ and the Banach space $Y$ be as
described above. Then
\begin{equation*}
  \sigma(T,Y) = \sigma(T,H^\infty(\mathds{U}) ;\ \text{and} \;  \sigma_i(T,Y) =
  \sigma_i(T,H^\infty(\mathds{U}), i= 1, \ldots , 5.
\end{equation*}
\end{corollary}

\end{example}

\begin{remark} \label{r4} If in Example~\ref{e6} we assume that $X$ is an interpolation
space between $L^\infty(\mathds{T}, m)$ and $L^1(\mathds{T}, m)$ then the results
stated in that example can be extended to operators of the form $w(y \circ \varphi)$
where $\varphi$ is an \textbf{elliptic} non-periodic automorphism of $\mathds{U}$.
\end{remark}

\begin{example} \label{e7} Let $n >1$ and $m_n$ be the normalized Lebesgue measure on
$\mathds{T}^n$. Let $X$ be a Banach ideal space in $L^0(\mathds{T}^n, m_n)$ such that
$L^\infty(\mathds{T}^n, m_n) \subseteq X \subseteq L^1(\mathds{T}^n, m_n)$ and the norm
on $X$ is \textbf{order continuous}. Notice that the last condition implies that $X$ is
an interpolation space. Let $Y = X \cap H^1(\mathds{U^n})$, $\varphi$ be a non-periodic
rotation of $\mathds{U}^n$, and $w \in H^\infty(\mathds{U}^n)$. Let $Ty = w(y \circ
\varphi),y \in Y$.

We claim that the conditions of Theorem~\ref{t16} are satisfied and, respectively, its
conclusions remain valid for operator $T$.

To see that $Y$ is almost localized in $X$ let $Z$ be a band in $X$. Then $Z = \chi_E
\mathds{T}^n$ where $m_n(E) > 0$. Let $C > 0$ be such that $\|C\chi_E\|_X = 1$. Fix
$\varepsilon >0$. Because the norm in $X$ is order continuous there is an open subset
$O_\varepsilon$ of $\mathds{T}^n$ such that $E \subseteq O_\varepsilon$, $m_n(\partial
O_\varepsilon) = 0$, and $\|C \chi_{O_\varepsilon} - C \chi_E\| < \varepsilon$.
Consider the function $g = C\chi_{O_\varepsilon} + \varepsilon \chi_{\mathds{T}^n
\setminus cl O_\varepsilon}$. Function $g$ is lower semicontinuous on $\mathds{T}^n$
and therefore (see~\cite[Theorem 3.5.3]{Ru1}) there is $y \in H^\infty(\mathds{U}^n)
\subset Y$ such that the boundary values of $y$ coincide a.e. on $\mathds{T}^n$ with
$g$.

Condition~(\ref{eq36}) follows from Theorem 3.3.5 in~\cite{Ru1}.

Under the conditions of this example we can also improve statement (2) of
Theorem~\ref{t16} by claiming that $\sigma(T) = \sigma_3(T)$. The reasoning is the same
as in Example~\ref{e4}

 Note that if $X$ is an interpolation space then instead of non-periodic rotations of
 $\mathds{U}^n$ we can consider more general transformations considered in
 Example~\ref{e4}.
\end{example}

\begin{example} \label{e8} Instead of polydisc considered in the previous example we
can consider Banach-Hardy spaces with order continuous norm on $\mathds{B}^n$ and
weighted composition operators of the form
$$ Tx = w(x \circ \tau), x \in X, $$
where $w \in H^\infty(\mathds{B}^n)$ and $\tau$ is a non-periodic unitary
transformation of $\mathds{C}^n$.
 In this case the condition that $Y$ is almost localized in $X$ follows
 from~\cite[Theorem 12.5]{Ru3} and condition~(\ref{eq36}) from~\cite[Theorem
 5.6.4]{Ru2}.
\end{example}

\begin{problem} \label{pr5} Is it possible in Examples~\ref{e7} and/or~\ref{e8} to
weaken the condition that the norm on $X$ is order continuous? Of course, of particular
interest is the case of weighted rotation operators on $H^\infty(\mathds{U}^n)$ and
$H^\infty(\mathds{B}^n)$.
\end{problem}

\begin{example} \label{e9} Let $\mathds{A}$ be an annulus in $\mathds{C}$ centered at
$0$. The Hardy spaces $H^p(\mathds{A})$ were considered by Kas'yanyuk in~\cite{Kas} and
by Sarason in~\cite{Sa}. Following~\cite{Sa} we consider the annulus
$$ \mathds{A} = \{z \in \mathds{C} : 0 < R < |z| < 1 \}.$$
Let $1 \leq p < \infty$. Then
$$ H^p(\mathds{A}) = \{f : f \; \text{is holomorphic in} \; \mathds{A} \; \text{and}
\sup \limits_{R < r < 1} \int \limits_0^{2\pi} |f(re^{i\theta}|^p d\theta < \infty \}
.$$
As usual, $H^\infty(\mathds{A})$ means the algebra of all bounded analytic functions in
$\mathds{A}$.

It is proved in~\cite{Sa} that $H^p(\mathds{A}), 1 \leq p \leq \infty$, can be
identified with a closed subspace of $L^p(\partial \mathds{A})$. Moreover, it follows
from Theorem 9 in~\cite{Sa} that the space $H^p(\mathds{A})$ is almost localized in
$L^p(\partial \mathds{A})\breve{}$ and that condition~(\ref{eq36}) is satisfied.

Let $w \in H^\infty(\mathds{A})$ and $\varphi$ be a non-periodic rotation of
$\mathds{A}$. Let us fix $p \in [1, \infty]$ and let
$$ Tx = w(x \circ \varphi), x \in H^p(\mathds{A}). $$
It follows from Theorems~\ref{t16} and~\ref{t9} that $\sigma(T)$ is a rotation
invariant connected subset of $\mathds{C}$. Moreover, we can add the following details.

\noindent (1) $\sigma_1(T) = \sigma(T, L^p(\mathds{T})) \cup \sigma(T,
L^p(R\mathds{T}))$. Therefore $\sigma_1(T)$ is either a connected rotation invariant
subset of $\mathcal{C}$ or the union of two rotation invariant disjoint connected
subsets.

\noindent (2) The set $\sigma_r(T)$ if it is nonempty, is the union of at most two open
disjoint rotation invariant components: $\mathcal{C}_1$ - an open disc centered at $0$
and $\mathcal{C}_2$ - an open annulus. If $\lambda \in \mathcal{C}_1$ then the
following conditions are equivalent

(a) $\lambda \in \sigma_4(T) \setminus \sigma_3(T)$,

(b) there are constants $\varepsilon, \delta > 0$ such that
$$z \in Int_{\mathds{C}}\mathds{A},\; dist(z, \partial \mathds{A}) < \varepsilon
\Rightarrow |w(z)| > \delta. $$

On the other hand, if $\lambda \in \mathcal{C}_2$, we do not know under what are
necessary and/or sufficient conditions for $\lambda \in \sigma_4(T) \setminus
\sigma_3(T)$ (cf. Problem~\ref{pr6}).

\end{example}

\bigskip

\subsection{\centerline{Bergman spaces}}

Recall that the Bergman space $\mathds{A}^p(\mathds{U})$, $1 \leq p < \infty$, consists
of all functions analytic in $\mathds{U}$ and such that
$$\int \limits_0^{2\pi} \int \limits_0^1 |x(re^{i\theta})|^p r dr d\theta < \infty.$$
Endowed with the norm
$$ \|x\| = \Bigg{(}\int \limits_0^{2\pi} \int \limits_0^1 |x(re^{i\theta})|^p r dr
d\theta \Bigg{)}^{1/p} $$
$\mathds{A}^p(\mathds{U})$ is a Banach space (see e.g.~\cite{DuS}). Let $\alpha \in
\mathds{T}$ be not a root of unity and let $w \in H^\infty(\mathds{U})$. Clearly the
operator
$$ (Tx)(z) = w(z)x(\alpha z), x \in \mathds{A}^p(\mathds{U}), z \in \mathds{U}$$
is defined and bounded on $\mathds{A}^p(\mathds{U})$. Vice versa, every multiplier of
$\mathds{A}^p(\mathds{U})$ belongs to $H^\infty(\mathds{U})$ (see~\cite[Theorem 12, p. 59]{DuS}).

\begin{proposition} \label{p1} The operator $Z, (Zx)(z) = zx(z), x \in
\mathds{A}^p(\mathds{U})$ is invertible from the left and therefore by
Theorems~\ref{t5} and~\ref{t26} the sets $\sigma(T)$ and $\sigma_{a.p.}(T) =
\sigma_2(T) = \sigma_1(T)$ are rotation invariant.
\end{proposition}

\begin{proof} The proof follows from the definition of norm in
$\mathds{A}^p(\mathds{U})$ and the fact that $\int \limits_0^{2\pi} |x(re^{i\theta})|^p
d\theta$ is a nondecreasing function of $r$ on $[0,1)$ (see~\cite[Theorem 1.5, p.
9]{Du}).
\end{proof}

\begin{proposition} \label{p2} The spectrum $\sigma(T)$ is a circle, annulus, or disc
centered at $0$.
\end{proposition}

\begin{proof} First notice that by Theorem~\ref{t4} we have
$$\rho(T) \geq \exp \int \limits_0^{2\pi} \ln |w(e^i\theta)| d \theta > 0. $$
By Proposition~\ref{p1} it remains to prove that $\sigma(T)$ is connected. If not, then
there is a positive $r$ such that $\sigma(T) = \sigma_1 \cup \sigma_2$ where $\sigma_1
\subset \{z \in \mathds{C} : |z| < r\}$ and $\sigma_2 \subset \{z \in \mathds{C} : |z|
> r\}$. Let $P_1$ and $P_2$ be the corresponding spectral projections. Let $x_0 = P_1
1$ and for any $f \in H^\infty(\mathds{U})$ let $Fx = fx, x \in
\mathds{A}^p(\mathds{U})$. By Theorem~\ref{t9} $P_1$ commutes with $F$ and therefore
$P_1x = xx_0, x \in H^\infty(\mathds{U}) \subset \mathds{A}^p(\mathds{U})$. Because
$H^\infty(\mathds{U})$ is dense in $\mathds{A}^p(\mathds{U})$ we get that $P_1$ is the
operator of pointwise multiplication by $x_0$ and therefore $x_0 \in
H^\infty(\mathds{U})$. But then $x_0$ is a nonzero idempotent in $H^\infty(\mathds{U})$
and hence $x_0 = 1$, a contradiction.
\end{proof}

Propositions~\ref{p1} and~\ref{p2} give some information about the essential spectra of
weighted rotations on Bergman spaces but fall short of providing a complete description
of these spectra.

\begin{problem} \label{pr7} Let $w \in H^\infty(\mathds{U})$, $\alpha \in \mathds{T}$
be not a root of unity, and
$$ (Tx)(z) = w(z)x(\alpha z), x \in \mathds{A}^p(\mathds{U}), z \in \mathds{U}.$$
Describe essential spectra of $T$.
\end{problem}

The theorem below provides a partial solution of Problem~\ref{pr7}
under the additional assumption that the weight $w$ belongs to the
disc-algebra $A(\mathds{U})$.

\begin{theorem} \label{t17} Let $w \in A(\mathds{U})$ and let $w = B w_e$ be the
canonical factorization of $w$. Then
\begin{enumerate}
  \item If $w$ is invertible in $A(\mathds{U})$ then $\sigma(T) = \sigma_1(T) =
      |w(0)|\mathds{T}$.
  \item If $w$ is invertible in $C(\mathds{T})$ but not invertible in $A(\mathds{U})$
      then
  $$\sigma_3(T) = \sigma_1(T) = \sigma_{a.p.}(T) = |w_e(0)|\mathds{T}$$
   and
   $$\sigma_4(T) \setminus \sigma_3(T) = \sigma_r(T) = |w_e(0)|\mathds{U}$$.
  \item If $w$ is not invertible in $C(\mathds{T})$ then $\sigma_1(T) = \sigma(T) =
      |w_e(0)|\mathds{D}$.
\end{enumerate}

\end{theorem}

\begin{proof} (1) Follows immediately from Theorem~\ref{t4}.

\noindent (2) By Proposition~\ref{p2} and Theorem~\ref{t4} $\sigma(T) =
|w_e(0)|\mathds{D}$. Let $\lambda \in |w_e(0)|\mathds{U}$. We claim that $\lambda \in
\sigma_r(T)$. Indeed, let us fix $s$, $|\lambda| < s < \rho(T)$. It follows easily from
Theorem~\ref{t4} and the fact that $w$ is invertible in $C(\mathds{T})$ and continuous
on $\mathds{D}$ that there are an $n \in \mathds{N}$ and an $R \in (0,1)$ such that
\begin{equation}\label{eq38}
   |w_n(z)| > s^n, \forall z : R \leq |z| \leq 1.
\end{equation}
Let $x \in \mathds{A}^p(\mathds{U})$, $\|x\| = 1$. Then it follows  from~(\ref{eq38})
and the fact that $\int \limits_0^{2\pi} |x(re^{i\theta})|^p d\theta$ is a
nondecreasing function of $r$ on $[0,1)$ that
\begin{multline*}
\|T^n x - \lambda^n x\| \geq (s^n - |\lambda|^n) \Bigg{(}\int \limits_0^{2\pi} \int
\limits_R^1 |x(re^{i\theta})|^p r dr d\theta \Bigg{)}^{1/p} \\
 \geq (1 - R)^{1/p}(s^n - |\lambda|^n).
\end{multline*}

It remains to prove that for $\lambda \in |w_e(0)|\mathds{U}$ we have $\mathrm{codim}
\; (\lambda I - T)\mathds{A}^p(\mathds{U}) < \infty$. Using the stability of the index
we see that it is sufficient to prove that
$\mathrm{codim}\; T \mathds{A}^p(\mathds{U})< \infty$. Notice that $w = B w_e$ where
$B$ is a finite Blaschke product and $w_e$ is
invertible in $A(\mathds{U})$. Therefore it is sufficient to notice that
$\mathrm{codim} \; B\mathds{A}^p(\mathds{U}) < \infty$, as follows
from~\cite[Proposition 1]{AB}.

\noindent (3) By Theorem~\ref{t26} it is sufficient to prove that $\sigma_{a.p.}(T) =
|w_e(0)|\mathds{D}$. Let $\lambda \in |w_e(0)|\mathds{D}$. Then by Theorem~\ref{t1}
$\lambda \in \sigma_{a.p.}(T,C(\mathds{T}))$. Without loss of generality we can assume
that $\lambda = 1$. By Lemma~\ref{l4} there is a $k \in \mathds{T}$ such that
\begin{equation} \label{eq39}
|w_n(k)| \geq 1\;  \text{and}\; |w_n(\varphi^{-n}(k))| \leq 1, n \in \mathds{N},
\end{equation}
where $\varphi(t) = \alpha t, t \in \mathds{T}$.
Let $q(z) = \frac{1}{2}(z+k), z \in \mathds{U}$ and let $Q_n =
q^n/\|q^n\|_{\mathds{A}^p(\mathds{D})}, n \in \mathds{N}$. Then
(see~\cite[proof of Lemma 5, p.130]{DuS})
\begin{equation}\label{eq40}
  Q_n(z) \rightarrow 0 \; \text{uniformly on} \; \mathds{D} \setminus V,
\end{equation}
 where $V$ is an arbitrary open neighborhood of $k$ in $\mathds{D}$. For any $m \in
 \mathds{N}$ let $V_m$ be an open neighborhood of $k$ in $\mathds{D}$ such that the
 sets $\alpha^j V_m, |j| \leq m+1$, are pairwise disjoint.

Let us fix $m$ and let
\begin{equation*}
  F_n =\frac{1}{w_m(k)} U^{-m}Q_{n} \; \text{and} \; G_n = S_m(T,1/\sqrt{m})F_n.
\end{equation*}
where $(Ux)(z) = x(\alpha z), x \in \mathds{A}^p(\mathds{U}), z \in \mathds{U}$.

\noindent Then it follows from~(\ref{eq40}) and~(\ref{eq18}) that
\begin{multline}\label{eq48}
  \lim \limits_{n \rightarrow \infty} \|G_n\| = \sum \limits_{j=0}^m \Big{(}1 - \frac{1}{\sqrt{m}}\Big{)}^{n-j} \frac{1}{|w_j(k)|} + \\
   \sum \limits_{j=1}^{n-1} \Big{(}1 - \frac{1}{\sqrt{m}}\Big{)}^j |w_j(\varphi^{-j}(k)|.
\end{multline}
From~(\ref{eq48}) and~(\ref{eq19}) follows that
\begin{multline}\label{eq49}
  \limsup \limits_{n \rightarrow \infty} \|TG_n - G_n\| \leq \frac{1}{\sqrt{m}} \|G_n\| + \\
   \Big{(} 1 - \frac{1}{\sqrt{m}} \Big{)}^m \Big{(} \frac{1}{|w_m(k)|} + |w_m(\varphi^{-m}(k)| \Big{)}
\end{multline}
Because $m$ is arbitrary large and in virtue of~(\ref{eq39}) it follows from~(\ref{eq49}) that $1 \in \sigma_{a.p.}(T)$.
 \end{proof}

 For an open subset $\Omega$ of $\mathds{C}^n$ and for $p \in [1, \infty)$ the Bergman
 space $\mathds{A}^p(\Omega)$ is defined as
 \begin{equation*}
   \{x \in \mathcal{H}(\Omega): \int |f|^p dV < \infty\},
 \end{equation*}
where $V$ is the volume in $\mathds{R}^{2n}$. It is well known and easy to prove that
endowed with the norm
\begin{equation*}
  \Big{(} \int |f|^p dV \Big{)}^{1/p}
 \end{equation*}
$\mathds{A}^p(\Omega)$ is a Banach space.

Analogues of Theorem~\ref{t17} can be proved for weighted rotation operators in spaces
$\mathds{A}^p(\mathds{U}^n)$ and $\mathds{A}^p(\mathds{B}^n)$, $n > 1$. We will state
the
corresponding results and outline the changes that have to be made in the proof of
Theorem~\ref{t17}.

\begin{theorem} \label{t27} Let $p \in [1, \infty)$, $w \in A(\mathds{U}^n)$, and let
$T \in
L(\mathds{A}^p(\mathds{U}^n)$ be defined as
 $$ (Tx)(\mathbf{z}) = w(\mathbf{z})x(\boldsymbol{\alpha} \mathbf{z}), w \in
 A(\mathds{U}^n), x \in \mathds{A}^p(\mathds{U}^n), \mathbf{z} \in \mathds{U}^n,$$
 where $\boldsymbol{\alpha} =(\alpha_1, \ldots , \alpha_n) \in \mathds{T}^n$ is a
 non-periodic rotation of $\mathds{U}^n$. Then

 \noindent (1) If $w$ is invertible in $A(\mathds{U}^n)$ then $\sigma_1(T) = \sigma(T)$
 is either a circle  or an annulus centered at $0$. In particular, if
 condition~(\ref{eq25}) is satisfied then
 \begin{equation*}
   \sigma(T) = \rho(T)\mathds{D}, \;\text{and} \; \rho(T) = \int \limits_{\mathds{T}^n}
   \ln |w| d m_n.
 \end{equation*}

 \noindent (2) If $w$ is invertible in $C(\mathds{T}^n)$ but not invertible in
 $A(\mathds{U}^n)$ then $\sigma_1(T) = \sigma_{a.p.}(T)$ is a circle or an annulus,
 $\sigma_r(T) =  s\mathds{U}$  where  $s = \rho(T^{-1}, C(\mathds{T}^n))^{-1}$ and $\sigma_3(T) =
 \sigma(T)= \rho(T)\mathds{D}$.

 \noindent (3) If $w$ is not invertible in $C(\mathds{T}^n)$ then $\sigma_1(T) =
 \sigma(T) = \rho(T)\mathds{D}$.
\end{theorem}

\begin{proof} It is sufficient to prove that $\sigma_{a.p.}(T, C(\mathds{T}^n) \subseteq \sigma_{a.p.}(T)$. The inverse inclusion and the rest of the statements of the theorem then would follow from Theorems\ref{t1}. Let
\begin{equation*}
  \varphi(\mathbf{t}) = \boldsymbol{\alpha}\mathbf{t} = (\alpha_1 t_1, \ldots ,
  \alpha_n t_n), \boldsymbol{t} \in \mathds{T}^n.
\end{equation*}
Let $\mathbf{k} \in \mathds{T}^n$ be such that inequalities~(\ref{eq39}) hold. Because
$\varphi$-periodic points are nowhere dense in $\mathds{T}^n$ we can for any $j \in
\mathds{N}$ find an open subset $V_j$ of $\mathds{T}^n$ such that the sets $cl
\varphi^l(V_j), |l| \leq j+1$ are pairwise disjoint and
\begin{equation} \label{eq46}
|w_m(\mathbf{t})| \geq 1/2 \;  \text{and}\; |w_m(\varphi^{-m}(\mathbf{t}))| \leq 2, |m|
\leq j+1, \mathbf{t} \in V_j.
\end{equation}

  Let $\mathbf{k}_j = (k_{1j}, \ldots, k_{nj}) \in V_j$ and let
\begin{equation*}
  q_j(\mathbf{z}) = \prod \limits_{i=1}^n \big{(}\frac{z_i + k_{ij}}{2}\big{)},
   q_j^m(\mathbf{z}) = [q_j(\mathbf{z})]^m, \mathbf{z}  \in \mathds{U}^n, m \in \mathds{N},
  \end{equation*}
   and let
 \begin{equation*}  Q_j^m = q_j^m/\|q_j^m\|_{\mathds{A}^p(\mathds{U}^n)}.
\end{equation*}
 It follows from the computation in~\cite[Proof of Lemma 5, p.130]{DuS} that for any $j \in \mathds{N}$
 \begin{equation}\label{eq43}
   \|q_j^m\|_{\mathds{A}^p(\mathds{U}^n)}^p \mathop \thicksim \limits_{m \rightarrow \infty} Cm^{-3n/2}.
 \end{equation}
  Therefore $Q_j^m(z) \mathop \rightarrow \limits_{m \to \infty} 0$ uniformly on $\mathds{D}^n \setminus V$ where $V$ is an arbitrary open neighborhood of
 $\mathbf{k}_j$ in $\mathds{D}^n$, and we can proceed as in the proof of part (3) of
 Theorem~\ref{t17}.
 \end{proof}

 \begin{theorem} \label{t28} Let $U$ be a unitary transformation of $\mathds{C}^n$ such
 that $U^n \neq I, n \in \mathds{N}$, let $w \in A(\mathds{B}^n)$ and let
$T \in L(\mathds{A}^p(\mathds{B}^n))$ be defined as
 $$ (Tx)(\mathbf{z}) = w(\mathbf{z})x(U\mathbf{z}), x \in \mathds{A}^p(\mathds{B}^n), \mathbf{z} \in \mathds{B}^n. $$
  Then

 \noindent (1) If $w$ is invertible in $A(\mathds{B}^n)$ then $\sigma_1(T) = \sigma(T)$
 is either a circle  or an annulus centered at $0$.

 \noindent (2) If $w$ is invertible in $C(\partial \mathds{B}^n)$ but not invertible in
 $A(\mathds{B}^n)$ then $\sigma_1(T) = \sigma_{a.p.}(T)$ is a circle or an annulus,
 $\sigma_r(T) =  s\mathds{U}$ where $s= \rho(T^{-1}, C(\partial \mathds{B}^n))^{-1}$, and $\sigma_3(T) = \sigma(T)= \rho(T)\mathds{D}$.

 \noindent (3) If $w$ is not invertible in $C(\partial \mathds{B}^n)$ then $\sigma_1(T) =
 \sigma(T) = \rho(T)\mathds{D}$.
 \end{theorem}

 \noindent \textit{Sketch of the proof}. Let $\varphi(\mathbf{z}) = U\mathbf{z}, \mathbf{z} \in cl \mathds{B}^n$. Let $\mathbf{k} \in \partial \mathds{B}^n$ be such that inequalities~(\ref{eq39}) hold. For simplicity we assume that $\mathbf{k}$ is not $\varphi$-periodic (if it is $\varphi$-periodic, we will apply the same procedure as in the proof of Theorem~\ref{t27}). Without loss of generality we can assume that $\mathbf{k} = (1, 0, \ldots , 0)$. Let
 \begin{equation*}
   q(\mathbf{z})= \frac{1+z_1}{2} \; \text{and} \; q^m(\mathbf{z}) = [q(\mathbf{z})]^m, \mathbf{z} \in cl \mathds{B}^n, m \in \mathds{N}.
 \end{equation*}
 Fix $j \in \mathds{N}$ and take an open neighborhood $V$ of $\mathbf{k}$ in $cl \mathds{B}^n$ such that the sets $\varphi^s(V), |s|\leq j+1,$ are pairwise disjoint. There is an $M \in (0,1)$ such that $|q(\mathbf{z})| < M$ on $cl \mathds{B}^n \setminus V$. Let us fix an $\varepsilon \in (0,1)$ and notice that

\begin{multline}\label{eq47}
 \|q^m\|^p = \frac{\pi^n}{n!} \iint \limits_{\mathds{D}} \Big{|}\frac{1+z_1}{2}\Big{|}^{mp} (1 -|z_1|^2)^n dA = \\
 = \frac{\pi^n}{n!2^{mp}} \int \limits_0^{2\pi} \int \limits_0^1 (1+2r\cos{\theta} +r^2)^{\frac{mp}{2}}r(1 -r^2)^n dr d\theta \geq \\
 \frac{\pi^n}{n!2^{mp}} \int \limits_0^{\arccos{(1- \varepsilon)}} \int \limits_0^1 (1+2r\cos{\theta} +r^2)^{\frac{mp}{2}}r(1 -r^2)^n dr d\theta \geq \\
  \frac{\pi^n}{n!2^{mp}} \arccos{(1 - \varepsilon)} \int \limits_0^1 (1+ (3 - \varepsilon) r^2)^{\frac{mp}{2}}r(1 -r^2)^n dr = \\
  = \frac{\pi^n}{2n!2^{mp}} \arccos{(1 - \varepsilon)} \int \limits_0^1 (1+ (3 - \varepsilon) u)^{\frac{mp}{2}}(1 - u)^n du
 \end{multline}
 Applying integration by parts $n$ times to the last integral in~(\ref{eq47}) we can see that
 \begin{equation}\label{eq50}
 \|q^m\|^p \geq \frac{c}{(mp+n)^{n+1}} \Big{(} \frac{4- \varepsilon}{4} \Big{)}^{mp/2},
 \end{equation}
where the constant $c, c > 0$, does not depend on $m$. Taking an $\varepsilon$ such that $\frac{4- \varepsilon}{4}> M$ and considering $Q_n = q_n/ \|q_n\|$ we obtain from~(\ref{eq50}) that $Q_n^m \mathop \rightarrow \limits_{m \to \infty} 0$ uniformly on $cl \mathds{B}^n \setminus V$. The remaining part of the proof repeats verbatim the corresponding part of the proof of Theorem~\ref{t17}. $\square$

 \bigskip

 \subsection{\centerline{The Bloch space}}

 The Bloch space $\mathcal{B}$ consists of all functions analytic in $\mathds{U}$ and
 such that
 $$ \sup \limits_{z \in \mathds{U}} (1 - |z|^2)|x^\prime (z)| < \infty . $$
Endowed with the norm
$$ \|x\| = |x(0)| + \sup \limits_{z \in \mathds{U}} (1 - |z|^2)|x^\prime (z)| $$
$\mathcal{B}$ is a Banach space.
The little Bloch space $\mathcal{B}_0$ is the closure of polynomials in $\mathcal{B}$.
It is well known (see~\cite{DuS}) that
$$ \mathcal{B}_0 = \{x \in \mathcal{B} : \lim \limits_{|z| \rightarrow 1-} (1 -
|z|^2)|x^\prime (z)| =0 \}. $$
Let $\mathcal{M(B)}$ be the Banach space of all multipliers of $\mathds{B}$. It was
proved in~\cite[Theorem 1]{BS1} that
\begin{multline} \label{eq44} \mathcal{M(B)} = \mathcal{M(B_{\mathrm{0}})} =  \\
= \{w \in H^\infty(\mathds{U}) : \sup \limits_{z \in \mathds{U}}
(1 - |z|)\ln{(1/(1 - |z|))}|w^\prime (z)| < \infty\} .
\end{multline}
It follows from~(\ref{eq44}) and Theorem~\ref{t4} that if
$w \in \mathcal{M(B)}$, $\alpha \in \mathds{T}$ and
$$(Tx)(z) = w(z)x(\alpha z), x \in \mathcal{B}, z \in \mathds{U}, $$
then
\begin{equation}\label{eq51}
   \rho(T) = \max \limits_{\mu \in M_\varphi} \exp \int \ln |\hat{w}| d\mu,
\end{equation}
where $\varphi$ is the homeomorphism of $\mathfrak{M}(\mathcal{M}(\mathcal{B}))$ generated by the rotation of $\mathds{U}$ by $\alpha$ and $M_\varphi$ is the set of all $\varphi$-invariant regular probability Borel measures on $\partial
(\mathcal{M}(\mathcal{B}))$. In particular, it is not difficult to see that if $w \in A(\mathds{U}) \cap \mathcal{M(B)}$ then
$$ \rho(T) = |w_e(0)|. $$

\begin{remark} \label{r11} The second dual $(\mathcal{B}_0)^{\prime \prime}$ can be
canonically identified with $\mathcal{B}$ (see~\cite{DuS}). The operator $T$ can be
restricted on $\mathcal{B}_0$ and $(T|\mathcal{B}_0)^{\prime \prime} = T$. Therefore
$\sigma(T) = \sigma(T|\mathcal{B}_0)$, $\sigma_{a.p.}(T) =
\sigma_{a.p.}(T|\mathcal{B}_0)$, and $\sigma_j(T) = \sigma_j(T|\mathcal{B}_0), j = 1,
\ldots , 5$.
\end{remark}

\begin{theorem} \label{t21} Let
$$(Tx)(z) = w(z)x(\alpha z), x \in \mathcal{B}, z \in \mathds{U}, $$
where $\alpha \in \mathds{T}$ is not a root of unity and $w \in \mathcal{M(B)}$.

\noindent Then the sets $\sigma_{a.p.}(T) = \sigma_1(T)$, and therefore  $\sigma(T)$ are rotation invariant. Moreover the set $\sigma(T)$ is connected.
\end{theorem}

\begin{proof} The connectedness of $\sigma(T)$ follows from the description of $\mathcal{M}(\mathcal{B})$ and Theorem~\ref{t9}. Let us prove that  $\sigma_{a.p.}(T)$ is rotation invariant.

 (I) Assume first that $\lambda \neq w(0)$. Let $x_n \in \mathcal{B}$,
$\|x_n\| = 1$, and $Tx_n - \lambda x_n \rightarrow 0$. Then $x_n(0) \rightarrow 0$ and,
considering if necessary $y_n = x_n - x_n(0)1$, we can assume that $x_n(0) = 0$. Let
the points $z_n \in \mathds{U}$ be such that $x_n^\prime(z_n)|(1 - |z_n|^2)
= 1-\frac{1}{n}$. We need to consider two possibilities.

(a)$ \liminf \limits_{n \to \infty} |z_n| < 1$. Then by Montel's theorem there is a
subsequence $x_{n_k}$ that converges uniformly on compact subsets of $\mathds{U}$ to a
nonzero function $x$ analytic in $\mathds{U}$. It is immediate that $x \in \mathds{B}$
and $Tx = \lambda x$. Let $x_p(z) = z^p x(z), p \in \mathds{N}$. Then $Tx_p = \alpha^p
\lambda x_p$ and $\lambda \mathds{T} \subset \sigma_{a.p.}(T)$

(b) $\lim \limits_{n \to \infty} |z_n| = 1$. Let $y_n(z) = zx_n(z)$. If we can prove
that
\begin{equation}\label{eq45}
  \|y_n\| \geq c > 0,
\end{equation}
 then like in proof of Theorem~\ref{t5} we will obtain that $\lambda \mathds{T}
 \subseteq \sigma_{a.p.}(T)$. To prove~(\ref{eq45}) notice that
 $$ \|y_n\| \geq |x_n(z_n) +z_n x^\prime(z_n)|(1 - |z_n|^2) \geq (1-\frac{1}{n})|z_n| -
 |x_n(z_n)|(1 - |z_n|^2) \geq$$
 $$ (1-\frac{1}{n})|z_n| - \frac{1}{2}\ln \frac{1 - |z_n|}{1+|z_n|}(1 - |z_n|^2)
 \mathop \rightarrow \limits_{n \to \infty} 1.$$

(II) We turn to the case when $\lambda = w(0) \in \sigma_{a.p.}(T)$. Let $x_n \in
\mathcal{B}_0$, $\|x_n\| =1$, and $Tx_n - w(0)x_n \rightarrow 0$. If $ \liminf
\limits_{n \to \infty} |x_n(0)| = 0$ we can proceed as in part (I) of the proof. If, on
the other hand, $ \liminf \limits_{n \to \infty} |x_n(0)| >0$ then like in part (Ia) we
see that there is a nonzero $x \in \mathcal{B}$ such that $Tx = w(0)x$, and we are
done.
\end{proof}

We have the following analogue of Theorem~\ref{t17}

\begin{theorem} \label{t29} Let $\alpha \in \mathds{T}$ be not a root of unity,
$w \in A(\mathds{U}) \cap \mathcal{M}(\mathcal{B})$, and $T$ be the weighted rotation operator on $\mathcal{B}$,
\begin{equation*}
  (Tx)(z) = w(z)x(\alpha z), x \in \mathcal{B}, z \in \mathds{U}.
\end{equation*}
Then
\begin{enumerate}
  \item If $w$ is invertible in $A(\mathds{U})$ then $\sigma(T) = \sigma_1(T) =
      |w(0)|\mathds{T}$.
  \item If $w$ is invertible in $C(\mathds{T})$ but not invertible in $A(\mathds{U})$
      then
  $$\sigma_3(T) = \sigma_1(T) = \sigma_{a.p.}(T) = |w_e(0)|\mathds{T}$$
   and
   $$\sigma_4(T) \setminus \sigma_3(T) = \sigma_r(T) = |w_e(0)|\mathds{U}$$.
  \item If $w$ is not invertible in $C(\mathds{T})$ then $\sigma_1(T) = \sigma(T) =
      |w_e(0)|\mathds{D}$.
\end{enumerate}
\end{theorem}

\begin{proof} W will first prove the inclusion
\begin{equation}\label{eq52}
  \sigma_{a.p.}(T, C(\mathds{T})) \subseteq \sigma_{a.p.}(T).
\end{equation}
Let $\lambda \in \sigma_{a.p.}(T, C(\mathds{T}))$. We assume without loss of generality that $\lambda = 1$. Let $k \in \mathds{T}$ be such that inequalities~(\ref{eq39}) hold. Because $\mathcal{B}$ is rotation invariant we can assume that $k = 1$. Fix $j \in \mathds{N}$ and let $V$ be an open neighborhood of $1$ in $\mathds{D}$ such that inequalities~(\ref{eq46}) hold. Let
\begin{equation*}
q_m(z) = \Big{(} \frac{1+z}{2} \Big{)}^m, z \in \mathds{D}, m \in \mathds{N}.
\end{equation*}
Simple calculations show that
\begin{equation}\label{eq53}
  \|q_m\| = \frac{1}{2^m} + q_m^\prime \Big{(} \frac{m-1}{m+1} \Big{)} \Big{[} 1 - \Big{(} \frac{m-1}{m+1}\Big{)}^2 \Big{]} \mathop \thicksim \limits_{m \to \infty} \frac{4e^{-1}}{m}.
\end{equation}
Let $Q_m = q_m/\|q_m\|$.  It follows from~(\ref{eq53}) that
\begin{equation*}
 Q_m(z) + Q_m^\prime(z)(1 - |z|^2) \rightarrow 0 \; \text{uniformly on} \; \mathds{D} \setminus V,
\end{equation*}
and we can proceed as in the proof of Theorem~\ref{t17}.
\end{proof}

\begin{remark} \label{r9} (1) It follows from~\cite[Theorem 4.1]{AC} that if $w \in
A(\mathds{U}) \cap \mathcal{M(B)}$ then $w$ is invertible in $\mathcal{M(B)}$ if and
only if it is invertible in $A(\mathds{U})$.

(2) Analogues of Theorem~\ref{t29} can be proved for weighted rotation operators on Bloch spaces in polydisc and in the unit ball of $\mathds{C}^n$.
\end{remark}

\begin{problem} \label{pr9} Describe completely $\sigma_{a.p.}(T)$ for an arbitrary
weight $w \in \mathcal{M}(\mathcal{B})$. In particular, is it true that
$\sigma_{a.p.}(T) = \sigma_{a.p.}(\tilde{T})$ where $\tilde{T}$ is the operator on $H^\infty(\mathds{U})$ defined by the same formula as $T$?
\end{problem}

\bigskip

\subsection{The Dirichlet space}

The Dirichlet space $\mathcal{D}_2$ is the space of all functions analytic in
$\mathds{U}$ and such that
\begin{equation}\label{eq54}
  \|x\| = \Bigg{(}|x(0)|^2 + \int \limits_0^{2\pi} \int \limits_0^1
|x^\prime(re^{i\theta})|^2 r dr d \theta \Bigg{)}^{1/2} < \infty .
\end{equation}
Endowed with norm~(\ref{eq54}) $\mathcal{D}_2$ is a Hilbert space. A complete description of multipliers of $\mathcal{D}_2$ is not trivial and involves Carleson measures. The
interested reader is referred to~\cite{St} or~\cite{EK}.

We will consider $\mathcal{D}_2$ as a member of the scale of spaces $\mathcal{D}_p$, $1
\leq p < \infty$ where $\mathcal{D}_p$ is the space of all functions analytic in
$\mathds{U}$ and such that
\begin{equation}\label{eq55}
  \|x\| = \Bigg{(}|x(0)|^p + \int \limits_0^{2\pi} \int \limits_0^1
|x^\prime(re^{i\theta})|^p r dr d \theta \Bigg{)}^{1/p} < \infty .
\end{equation}
It follows easily from~(\ref{eq55}) that the norm on the Banach algebra $\mathcal{M}(\mathcal{D}_p)$ satisfies~(\ref{eq1}).
The proofs of the next two results are analogous to those of Theorem~\ref{t21} and Theorem~ and we omit them.

\begin{theorem} \label{t24} Let $1 \leq p < \infty$ and
$$(Tx)(z) = w(z)x(\alpha z), x \in \mathcal{B}, z \in \mathds{U}, $$
where $\alpha \in \mathds{T}$ is not a root of unity and $w \in
\mathcal{M}(\mathcal{D}_p)$.

\noindent Then the sets $\sigma(T)$ and $\sigma_{a.p.}(T) = \sigma_1(T)$ are rotation
invariant. Moreover, the set $\sigma(T)$ is connected.
\end{theorem}

\begin{theorem} \label{t30} Let $1 \leq p < \infty$. Let $\alpha \in \mathds{T}$ be not a root of unity,
$w \in A(\mathds{U}) \cap \mathcal{M}(\mathcal{D}_p)$, and $T$ be the weighted rotation operator on $\mathcal{D}_p$,
\begin{equation*}
  (Tx)(z) = w(z)x(\alpha z), x \in \mathcal{D}_p, z \in \mathds{U}.
\end{equation*}
Then
\begin{enumerate}
  \item If $w$ is invertible in $A(\mathds{U})$ then $\sigma(T) = \sigma_1(T) =
      |w(0)|\mathds{T}$.
  \item If $w$ is invertible in $C(\mathds{T})$ but not invertible in $A(\mathds{U})$
      then
  $$\sigma_3(T) = \sigma_1(T) = \sigma_{a.p.}(T) = |w_e(0)|\mathds{T}$$
   and
   $$\sigma_4(T) \setminus \sigma_3(T) = \sigma_r(T) = |w_e(0)|\mathds{U}$$.
  \item If $w$ is not invertible in $C(\mathds{T})$ then $\sigma_1(T) = \sigma(T) =
      |w_e(0)|\mathds{D}$.
\end{enumerate}

\end{theorem}

\bigskip

\subsection{Some spaces of analytic functions smooth in $\mathds{D}$}

Let $n \in \mathds{N}$ and let $Y$ be a Banach ideal space, $L^\infty(\mathds{T}) \subseteq X \subseteq L^1(\mathds{T})$. We will assume that the norm on $Y$ is rotation invariant. For $x \in \mathcal{H}(\mathds{U})$ we denote by $x^{(n)}$ the derivative of $x$ of order $n$. It is also convenient to put $x^{(0)} = x$. In this subsection we consider the following Banach spaces of functions analytic in $\mathds{U}$ .
\begin{equation}\label{eq56}
  C^n_A(\mathds{U}) = \{x \in \mathcal{H}(\mathds{U}): \; x^{(n)} \in A(\mathds{U})\}.
\end{equation}
Endowed with the norm
\begin{equation*}
  \|x\| = \sum \limits_{j=0}^{n-1} |x^{(j)}(0)| + \|x^{(n)}\|_{\infty}
\end{equation*}
$C^n_A(\mathds{U})$
 is a Banach algebra.
 \begin{equation}\label{eq57}
  W^{n,Y}_A(\mathds{U}) = \{x \in \mathcal{H}(\mathds{U}): \; x^{(n)} \in Y \cap  H^1(\mathds{U})\}.
\end{equation}
The norm on $W^{n,Y}_A(\mathds{U})$ is defined as
 \begin{equation*}
  \|x\| = \sum \limits_{j=0}^{n-1} |x^{(j)}(0)| + \|x^{(n)}\|_Y
\end{equation*}

\begin{theorem} \label{t31} Let $\alpha \in \mathds{T}$ be not a root of unity and let $w \in C^n_A(\mathds{U})$. Let
\begin{equation*}
  (Tx)(z) = w(z)x(\alpha z), x \in C^n_A(\mathds{U}), z \in \mathds{U},
\end{equation*}
and
\begin{equation*}
  (\tilde{T}x)(z) = w(z)x(\alpha z), x \in A(\mathds{U}), z \in \mathds{U}.
\end{equation*}
Then,
\begin{equation*}
  \sigma(T) = \sigma(\tilde{T}) \; \text{and} \; \sigma_i(T) = \sigma_i(\tilde{T}), i= 1, \ldots , 5.
\end{equation*}

\end{theorem}

\begin{theorem} \label{t32} Let $\alpha \in \mathds{T}$ be not a root of unity and let $w \in W^{n, L^\infty(\mathds{T})}_A(\mathds{U})$. Let
\begin{equation*}
  (Tx)(z) = w(z)x(\alpha z), x \in W^{n,Y}_A(\mathds{U}), z \in \mathds{U},
\end{equation*}
and
\begin{equation*}
  (\tilde{T}x)(z) = w(z)x(\alpha z), x \in Y \cap H^1(\mathds{U}), z \in \mathds{U}.
\end{equation*}
Then the operators $T$ and $\tilde{T}$ are bounded in $W^{n,Y}_A(\mathds{U})$  and $Y \cap  H^1(\mathds{U})$, respectively and
\begin{equation*}
  \sigma(T) = \sigma(\tilde{T}) \;, \sigma_i(T) = \sigma_i(\tilde{T}), i= 1, \ldots , 5.
\end{equation*}

\end{theorem}

\begin{remark} \label{r14} In virtue of Example~\ref{e4} and Corollary~\ref{c11} Theorems~\ref{t31} and~\ref{t32} provide a complete description of essential spectra of $T$ in $C^n_A(\mathds{U})$ and $W^{n,Y}_A(\mathds{U})$, respectively.
\end{remark}

The proofs of Theorems~\ref{t31} and~\ref{t32} are very similar and therefore we provide only (a little bit more complicated of two) proof of Theorem~\ref{t32}.

\textit{Proof of Theorem~\ref{t32}}. Let $W_0$ be the closed subspace of $W^{n,Y}_A(\mathds{U})$ defined as
\begin{equation*}
  W_0 = \{x \in W^{n,Y}_A(\mathds{U}) : x(0) = x^{(1)}(0) = \dots = x^{(n-1)}(0) = 0\}.
\end{equation*}
Clearly $TW_0 \subseteq W_0$ and $\mathrm{codim} \; W_0 = n$. The map $Jx = x^{(n)}, x \in W_0$ is a linear isometry of $W_0$ onto $Y$ and it is immediate to see that
\begin{equation}\label{eq58}
  JTJ^{-1}y = \sum \limits_{k=0}^n \binom{n}{k} \alpha^k w^{(n-k)}(J^{-1}y)^{(k)}.
\end{equation}
It follows from~(\ref{eq58}) that
\begin{equation}\label{eq59}
  JTJ^{-1}y = K + \alpha^n \tilde{T},
\end{equation}
where $K \in L(Y)$ is a compact operator. It follows from~(\ref{eq59}) and from $\mathrm{codim} \; W_0 = n$ that  $\sigma_i(T) = \sigma_i(\alpha^n \tilde{T}), i= 1, \ldots , 4$. Because by Corollary~\ref{c11} the sets $\sigma_i(\tilde{T}), i = 1, \ldots , 4$, are rotation invariant we see that $\sigma_i(T) = \sigma_i(\tilde{T}), i= 1, \ldots , 4$. Next notice that the operator $Z$, $(Zx)(z) = zx(z)$, is an isometry on $W_0$ and therefore by Theorem~\ref{t5} the set $\sigma(T|W_0)$ is rotation invariant. Hence, $\sigma_5(T|W_0) = \sigma_4(T|W_0) = \sigma_4(T)$.

It remains to notice that $\sigma(T) = \sigma(T|W_0)$. Indeed, if $\lambda \in \sigma(T) \setminus \sigma(T|W_0)$ then $\lambda$ is an isolated eigenvalue of $T$. Let $x$ be a corresponding eigenvector. Then  $x \in W^{n,Y}_A(\mathds{U}) \subset Y$ and we have $\lambda \in \sigma(\tilde{T}) = \sigma(T|W_0)$. $\square$

\bigskip

\subsection{The space $\ell^1_A$}
 In this last subsection we consider the Banach algebra $\ell^1_A$ of all functions analytic in $\mathds{U}$ with absolutely convergent Taylor series. In what follows $\alpha \in \mathds{T}$ is not a root of unity, $w \in \ell^1_A$, and
\begin{equation} \label{eq60}
  (Tx)(z) = w(z)x(\alpha z), x \in \ell^1_A, z \in \mathds{U}.
\end{equation}
The next proposition follows in a trivial way from Theorem~\ref{t5}, Corollary~\ref{c3}, and Theorem~\ref{t25}.

\begin{proposition} \label{p3} Let $T$ be defined by~(\ref{eq60}). Then
\begin{enumerate}
  \item $\rho(T) > 0$.
  \item $\sigma(T)$ is a rotation invariant connected subset of $\mathds{C}$.
  \item The sets $\sigma_1(T) = \sigma_{a.p.}(T)$ and $\sigma_r(T)$ are rotation invariant.
  \end{enumerate}
\end{proposition}

\noindent We can get more information about spectra of $T$ at the price of imposing an additional condition on the weight $w$. Consider the space $\Lambda$,
\begin{equation*}
  \Lambda = \{ x \in C(\mathds{T}): \int \limits_0^1 \omega(h)h^{-\frac{3}{2}}dh < \infty\},
\end{equation*}
where
\begin{equation*}
  \omega(h) = \sup \limits_{t_1,t_2 \in \mathds{T}, |t_1 - t_2|/2\pi \leq h}|x(t_1) - x(t_2)|
\end{equation*}
is the modulus of continuity of $x$.
It is easy to see that endowed with the norm
\begin{equation*}
  \|x\|_\Lambda = \|x\|_\infty + \int \limits_0^1 \omega(h)h^{-\frac{3}{2}}dh
\end{equation*}
$\Lambda$ is a Banach space. Moreover, $x, y \in \Lambda \Rightarrow xy \in \Lambda$ and
\begin{equation} \label{eq61}
  \|xy\|_\Lambda \leq \|x\|_\infty \|y\|_\Lambda + \|x\|_\Lambda \|y\|_\infty.
\end{equation}
By the well known theorem of Bernstein (see e.g.~\cite[p.13]{Ka}) $\Lambda \subset A(\mathds{T})$ where $A(\mathds{T})$ is the space of all functions on $\mathds{T}$ with absolutely convergent Fourier series.

\begin{theorem} \label{t33} Let $w \in \ell^1_A \cap \Lambda$. Then $\rho(T) = |w_e(0)|$. In particular,

\noindent if $w$ is invertible in $\ell^1_A$ then
\begin{equation*}
  \sigma(T) = \sigma_1(T) = w_e(0)\mathds{T},
\end{equation*}
otherwise
\begin{equation*}
  \sigma(T) = w_e(0)\mathds{D},
\end{equation*}
\end{theorem}

\begin{proof} From the inclusions $\ell^1_A \cap \Lambda \subset \ell^1_A \subset A(\mathds{U})$ follows that
\begin{multline*}
  |w_e(0)| = \rho(T, A(\mathds{U})) = \lim \limits_{n \to \infty} \|w_n\|_\infty \leq \rho(T) = \\
  = \lim \limits_{n \to \infty} \|w_n\|_{\ell^1_A } \leq \lim \limits_{n \to \infty} \|w_n\|_\Lambda = \rho(T, \Lambda).
\end{multline*}
It remains to notice that in virtue of~(\ref{eq61}) and Theorem~\ref{t4} we have
$\rho(T, \Lambda) = \rho(T, A(\mathds{U}))$.
\end{proof}

\begin{problem} \label{pr11}

\noindent (a) Assume conditions of Theorem~\ref{t33} and assume that $w$ is not invertible in $\ell^1_A$. Is it true that $\sigma_{a.p.}(T) = \sigma_{a.p.}(T, A(\mathds{U}))$?

\noindent (b) Does the formula $\rho(T) = |w_e(0)|$ remain true for an arbitrary $w \in \ell_A^1$?
\end{problem}

\end{document}